\documentclass[10pt]{amsart}
\pdfoutput=1
\usepackage{amssymb,amscd,graphicx,enumerate,epsfig,color}
\usepackage[bookmarks=true]{hyperref}

\newcommand{\Claim}[1]{\noindent\textbf{Claim. }\textit{#1}\\}
\newcommand{\ClaimN}[2]{\noindent\textbf{Claim #1. }\textit{#2}\\}
\newtheorem{theorem}{Theorem}[section]

\newtheorem{lemma}[theorem]{Lemma}

\theoremstyle{definition}
\newtheorem{definition}[theorem]{Definition}

\newcommand{\D}{\mathcal{D}}
\newcommand{\DV}{\mathcal{D}_\mathcal{V}}
\newcommand{\DW}{\mathcal{D}_\mathcal{W}}
\newcommand{\DVW}{\mathcal{D}_{\mathcal{VW}}}
\newenvironment{proofc}{\noindent\textit{Proof of Claim.}}{\\}
\newenvironment{proofN}[1]{\noindent\textit{Proof of #1.}}{\hfill$\square$\\}

\newcommand{\Case}[1]{\textbf{Case #1.}}

\newcommand{\V}{\mathcal{V}}
\newcommand{\W}{\mathcal{W}}

\begin{document}

\title[]{A topologically minimal, weakly reducible, unstabilized Heegaard splitting of genus three is critical}

\author{Jungsoo Kim}
\date{March 19, 2015}

\begin{abstract}
	Let $(\V,\W;F)$ be a weakly reducible, unstabilized, genus three Heegaard splitting in an orientable, irreducible $3$-manifold $M$.
	In this article, we prove that either the disk complex $\D(F)$ is contractible or $F$ is critical.
	Hence, the topological index of $F$ is two if $F$ is topologically minimal.
\end{abstract}

\address{\parbox{4in}{
	BK21 PLUS SNU Mathematical Sciences Division,\\ Seoul National University\\ 
	1 Gwanak-ro, Gwanak-Gu, Seoul 151-747, Korea}} 
	
\email{pibonazi@gmail.com}

\maketitle

\section{Introduction and Result}

Throughout this paper, all surfaces and 3-manifolds will be taken to be compact and orientable.
In \cite{Bachman2002}, Bachman introduced the concept a ``\textit{critical surface}'' and proved that a critical surface intersects an incompressible surface so that the intersection of them is essential on both surfaces up to isotopy in an irreducible manifold, where it is a common property for a strongly irreducible Heegaard surface (see \cite{Schultens2000}) and a critical surface.
In \cite{Bachman2008}, he generalized the definition of critical surface for the proof of Gordon's Conjecture by using the notations coming from the standard disk complex.
Moreover, he defined the concept a ``\textit{topologically minimal surface}'', which includes incompressible surfaces, strongly irreducible surfaces, critical surfaces, and so on, and the topologically minimal surfaces are  distinguished by the ``\textit{topological index}'' \cite{Bachman2010}.
Indeed, he proved that a topologically minimal surface also intersects an incompressible surface so that the intersection of them is essential on both surfaces up to isotopy in an irreducible manifold.
He also found the counterexamples of the Stabilization Conjecture by the method using this concept in \cite{Bachman2013}.
In \cite{4} \cite{5} \cite{6} , he proved that there is a resemblance between a topologically minimal surface and a geometrically minimal surface.

In recent results including the author's works, several examples of critical Heegaard surfaces were found and most of them are easily constructible \cite{BachmanJohnson2010} \cite{JHLee2013} \cite{LeiQiang2013} \cite{JungsooKim2012} \cite{JungsooKim2013}.
Hence, it is now guessed that it would be more easier for a weakly reducible surface to be topologically minimal than not to be topologically minimal.
Indeed, the condition that the disk complex is non-contractible for a topologically minimal surface seems to be more easier than the condition that the disk complex is contractible.

In this article, we will prove the following theorem giving an evidence to an affirmative answer for this question.

\begin{theorem}[Theorem \ref{theorem-main-a}]\label{theorem-main}
Let $(\V,\W;F)$ be a weakly reducible, unstabilized, genus three Heegaard splitting in an orientable, irreducible $3$-manifold $M$.
If every weak reducing pair of $F$ gives the same generalized Heegaard splitting after weak reduction and the embedding of each thick level in the relevant compression body is also unique up to isotopy, then the disk complex $\D(F)$ is contractible.
Otherwise, $F$ is critical.
\end{theorem}

Note that the author proved that if a weakly reducible, unstabilized, genus three Heegaard splitting in an orientable, irreducible $3$-manifold is topologically minimal, then the topological index is at most four in \cite{JungsooKim2013} and Theorem \ref{theorem-main} improves the upper bound of the topological index, i.e. the topological index is two if $F$ is topologically minimal.
Since there exist many unstabilized critical Heegaard surfaces of genus three, this upper bound is sharp.

\section{Preliminaries\label{section2}}

\begin{definition}
A \textit{compression body} (\textit{generalized compression body} resp.) is a $3$-manifold which can be obtained by starting with some closed, orientable, connected surface $F$, forming the product $F\times I$, attaching some number of $2$-handles to $F\times\{1\}$ and capping off all (some resp.) resulting $2$-sphere boundary components that are not contained in $F\times\{0\}$ with $3$-balls. 
The boundary component $F\times\{0\}$ is referred to as $\partial_+$. 
The rest of the boundary is referred to as $\partial_-$. 
\end{definition}

\begin{definition}
A \textit{Heegaard splitting} of a $3$-manifold $M$ is an expression of $M$ as a union $\V\cup_F \W$, denoted   as $(\V,\W;F)$ (or $(\V,\W)$ if necessary),  where $\V$ and $\W$ are compression bodies that intersect in a transversally oriented surface $F=\partial_+\V=\partial_+\W$. 
We say $F$ is the \textit{Heegaard surface} of this splitting. 
If $\V$ or $\W$ is homeomorphic to a product, then we say the splitting  is \textit{trivial}. 
If there are compressing disks $V\subset \V$ and $W\subset \W$ such that $V\cap W=\emptyset$, then we say the splitting is \textit{weakly reducible} and call the pair $(V,W)$ a \textit{weak reducing pair}. 
If $(V,W)$ is a weak reducing pair and $\partial V$ is isotopic to $\partial W$ in $F$, then we call $(V,W)$ a \textit{reducing pair}.
If the splitting is not trivial and we cannot take a weak reducing pair, then we call the splitting \textit{strongly irreducible}. 
If there is a pair of compressing disks $(\bar{V},\bar{W})$ such that $\bar{V}$ intersects $\bar{W}$ transversely in a point in $F$, then we call this pair a \textit{canceling pair} and say the splitting is \textit{stabilized}. 
Otherwise, we say the splitting is \textit{unstabilized}.
\end{definition}

\begin{definition}
Let $F$ be a surface of genus at least two in a compact, orientable $3$-manifold $M$. 
Then the \emph{disk complex} $\D(F)$ is defined as follows: 
\begin{enumerate}[(i)]
\item Vertices of $\D(F)$ are isotopy classes of compressing disks for $F$.
\item A set of $m+1$ vertices forms an $m$-simplex if there are representatives for each
that are pairwise disjoint.
\end{enumerate}
\end{definition}

\begin{definition}[Bachman, \cite{Bachman2010}]
The \emph{homotopy index} of a complex $\Gamma$ is defined to be 0 if $\Gamma=\emptyset$, and the smallest $n$ such that $\pi_{n-1}(\Gamma)$ is non-trivial, otherwise. 
We say a separating surface $F$ with no torus components is \emph{topologically minimal} if its disk complex $\D(F)$ is either empty or non-contractible. 
When $F$ is topologically minimal, we say its \emph{topological index} is the homotopy index of $\D(F)$. 
If $F$ is topologically minimal and its topological index is two, then we call $F$ a \textit{critical surface}. 
\end{definition}

Note that Bachman originally defined a \textit{critical surface} in a different way in \cite{Bachman2008} and proved it is equivalent to being index two in \cite{Bachman2010}.

\begin{definition}
Consider a Heegaard splitting $(\V,\W;F)$ of an orientable, irreducible $3$-manifold $M$. 
Let $\DV(F)$ and $\DW(F)$ be the subcomplexes of $\D(F)$ spanned by compressing disks in $\V$ and $\W$ respectively (see Chapter 5 of \cite{8}).
We call these subcomplexes \textit{the disk complexes of $\V$ and $\W$}.
Let $\DVW(F)$ be the subset  of $\D(F)$ consisting of the simplices having at least one vertex from $\DV(F)$ and at least one vertex from $\DW(F)$.
\end{definition}

\begin{theorem}[McCullough, Theorem 5.3 of \cite{8}]\label{theorem-mc}
$\DV(F)$ and $\DW(F)$ are contractible.
\end{theorem}

Note that $\D(F)=\DV(F)\cup \DVW(F)\cup \DW(F)$.

From now on, we will consider only unstabilized Heegaard splittings of an irreducible $3$-manifold. 
If a Heegaard splitting of a compact $3$-manifold is reducible, then the manifold is reducible or the splitting is stabilized (see \cite{SaitoScharlemannSchultens2005}).
Hence, we can exclude the possibilities of reducing pairs among weak reducing pairs.

\begin{definition}
Suppose $W$ is a compressing disk for $F\subset M$. 
Then there is a subset of $M$ that can be identified with $W\times I$ so that $W=W\times\{\frac{1}2\}$ and $F\cap(W\times I)=(\partial W)\times I$. 
We form the surface $F_W$, obtained by \textit{compressing $F$ along $W$}, by removing $(\partial W)\times I$ from $F$ and replacing it with $W\times(\partial I)$. 
We say the two disks $W\times(\partial I)$ in $F_W$ are the $\textit{scars}$ of $W$. 
\end{definition}

\begin{lemma}[Lustig and Moriah, Lemma 1.1 of \cite{7}] \label{lemma-2-8}
Suppose that $M$ is an irreducible $3$-manifold and $(\V,\W;F)$ is an unstabilized Heegaard splitting of $M$. 
If $F'$ is obtained by compressing $F$ along a collection of pairwise disjoint disks, then no $S^2$ component of $F'$ can have scars from disks in both $\V$ and $\W$. 
\end{lemma}

\begin{lemma}[J. Kim, Lemma 2.9 of \cite{JungsooKim2013}]\label{lemma-2-9}
Suppose that $M$ is an irreducible $3$-manifold and $(\V,\W;F)$ is a genus three, unstabilized Heegaard splitting of $M$. 
If there exist three mutually disjoint compressing disks $V$, $V'\subset\V$ and $W\subset \W$, then either $V$ is isotopic to $V'$, or one of $\partial V$ and $\partial V'$ bounds a punctured torus $T$ in $F$ and  the other is a non-separating loop in $T$. 
Moreover, we cannot choose three weak reducing pairs $(V_0, W)$, $(V_1,W)$, and $(V_2,W)$ such that $V_i$ and $V_j$ are mutually disjoint and non-isotopic in $\V$ for $i\neq j$. 
\end{lemma}

\begin{definition}[Definition 4.1 of \cite{Bachman2008}]\label{definition-2-10}
A \textit{generalized Heegaard splitting} (GHS) $\mathbf{H}$ of a $3$-manifold $M$ is a pair of sets of pairwise disjoint, transversally oriented, connected surfaces, $\operatorname{Thick}(\mathbf{H})$ and $\operatorname{Thin}(\mathbf{H})$ (called the \textit{thick levels} and \textit{thin levels}, respectively), which satisfies the following conditions.
\begin{enumerate}
\item Each component $M'$ of $M-\operatorname{Thin}(\mathbf{H})$ meets a unique element $H_+$ of $\operatorname{Thick}(\mathbf{H})$ and $H_+$ is a Heegaard surface in $M'$.
Henceforth we will denote the closure of the component of $M-\operatorname{Thin}(\mathbf{H})$ that contains an element $H_+\in\operatorname{Thick}(\mathbf{H})$ as $M(H_+)$.
\item As each Heegaard surface $H_+\subset M(H_+)$ is transversally oriented, we can consistently talk about the points of $M(H_+)$ that are ``above'' $H_+$ or ``below''  $H_+$.
Suppose $H_-\in\operatorname{Thin}(\mathbf{H})$.
Let $M(H_+)$ and $M(H_+')$ be the submanifolds on each side of $H_-$.
Then $H_-$ is below $H_+$ if and only if it is above $H_+'$.
\item There is a partial ordering on the elements of $\operatorname{Thin}(\mathbf{H})$ which satisfies the following: Suppose $H_+$ is an element of $\operatorname{Thick}(\mathbf{H})$, $H_-$ is a component of $\partial M(H_+)$ above $H_+$ and $H_-'$ is a component of $\partial M(H_+)$ below $H_+$.
Then $H_->H_-'$.
\end{enumerate}
We denote $\operatorname{Thin}(\mathbf{H})-\{\partial M\}$ as $\overline{\operatorname{Thin}}(\mathbf{H})$ and call it  the \textit{inner thin level}.
\end{definition}

\begin{definition}[Bachman, Definition 5.1 of \cite{Bachman2008}]
Let $M$ be a compact, orientable $3$-manifold.
Let $\mathbf{G}=\{T(\mathbf{G}),t(\mathbf{G})\}$ be a pair of sets of transversally oriented, connected surfaces in $M$ such that the elements of $T(\mathbf{G})\cup t(\mathbf{G})$ are pairwise disjoint.
Then we say $\mathbf{G}$ is a \textit{pseudo-GHS} if the following hold.
\begin{enumerate}
\item Each component $M'$ of $M-t(\mathbf{G})$ meets exactly one element $G_+$ of $T(\mathbf{G})$.
We denote the closure of $M'$ as $M(G_+)$.
\item Each element $G_+\in T(\mathbf{G})$ separates $M(G_+)$ into generalized (possibly trivial) compression bodies $\W$ and $\W'$, where $\partial_+\W=\partial_+\W'=G_+$.
\item There is a partial ordering of the elements of $t(\mathbf{G})$ that satisfies similar properties to the partial ordering of the thin levels of a GHS given in Definition \ref{definition-2-10}.
\end{enumerate}
\end{definition}

\begin{definition}[Bachman, a restricted version of Definition 5.2, Definition 5.3, and Definition 5.6 of \cite{Bachman2008}]
Let $M$ be a compact, orientable 3-manifold.
Let $\mathbf{H}$ be a Heegaard splitting of $M$, i.e. $\operatorname{Thick}(\mathbf{H})=\{F\}$ and $\operatorname{Thin}(\mathbf{H})$ consists of $\partial M$.
Let $V$ and $W$ be disjoint compressing disks of $F$ from the opposite sides of $F$ such that ${F}_{VW}$ has no $2$-sphere component. 
(Lemma 2.8 guarantees that ${F}_{VW}$ will not have a $2$-sphere component in the proof of Theorem \ref{theorem-main}.)
Define
$$T(\mathbf{G'})=(\operatorname{Thick}(\mathbf{H})-\{F\})\cup\{{F}_{V}, {F}_{W}\}, \text{ and}$$
$$t(\mathbf{G'})=\operatorname{Thin}(\mathbf{H})\cup\{{F}_{VW}\},$$
where we assume that each element of $T(\mathbf{G'})$ belongs to the interior of $\V$ or $\W$ by slightly pushing off $F_V$ or $F_W$ into the interior of $\V$ or $\W$ respectively and then also assume that they miss $F_{VW}$.
We say the pseudo-GHS $\mathbf{\mathbf{G'}}=\{T(\mathbf{G'}),t(\mathbf{G'})\}$ is obtained from $\mathbf{H}$ by \textit{pre-weak reduction} along $(V,W)$.
The relative position of the elements of $T(\mathbf{G'})$ and $t(\mathbf{G'})$ follows the order described in Figure \ref{figure-1}.
\begin{figure}
\includegraphics[width=10cm]{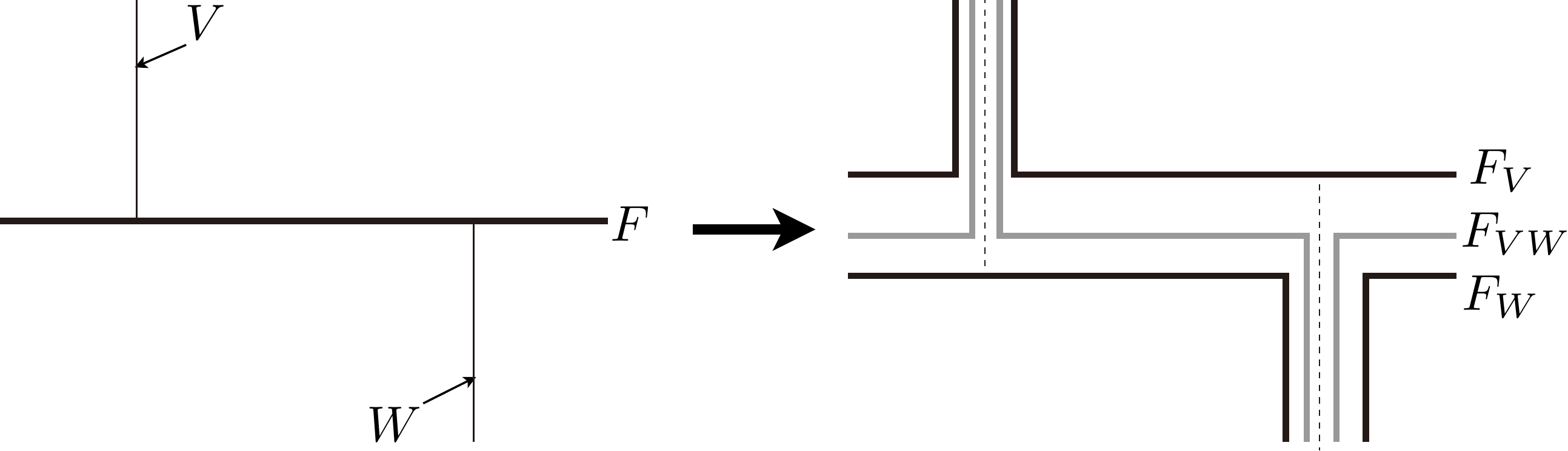}
\caption{pre-weak reduction \label{figure-1}}
\end{figure}
If there are elements $S\in T(\mathbf{G'})$ and $s\in t(\mathbf{G'})$ that cobound a product region $P$ of $M$ such that $P\cap T(\mathbf{G'})=S$ and $P\cap t(\mathbf{G'})=s$ then remove $S$ from $T(\mathbf{G'})$ and $s$ from $t(\mathbf{G'})$.
This gives a GHS $\mathbf{G}$ of $M$ from the pseudo-GHS $\mathbf{G'}$ (see Lemma 5.4 of \cite{Bachman2008}) and we say $\mathbf{G}$ is obtained from $\mathbf{G'}$ by \textit{cleaning}.
We say the GHS $\mathbf{G}$ of $M$ given by pre-weak reduction along $(V,W)$, followed by cleaning, is obtained from $\mathbf{H}$ by \textit{weak reduction} along $(V,W)$.
\end{definition}

\begin{definition}[J. Kim, \cite{JungsooKim2012}]
In a weak reducing pair for a Heegaard splitting $(\V,\W;F)$, if a disk belongs to $\V$, then we call it a \emph{$\V$-disk}.
Otherwise, we call it a \emph{$\W$-disk}.	
We call a $2$-simplex in $\DVW(F)$ represented by two vertices in $\DV(F)$ and one vertex in $\DW(F)$ a \textit{$\V$-face}, and also define a \textit{$\W$-face} symmetrically.
Let us consider a $1$-dimensional graph as follows.
\begin{enumerate}
\item We assign a vertex to each $\V$-face in $\DVW(F)$.
\item If a $\V$-face shares a weak reducing pair with another $\V$-face, then we assign an edge between these two vertices in the graph.
\end{enumerate}
We call this graph the \emph{graph of $\V$-faces}.
If there is a maximal subset $\varepsilon_\V$ of $\V$-faces in $\DVW(F)$ representing a connected component of the graph of $\V$-faces and the component is not an isolated vertex, then we call $\varepsilon_\V$ a \emph{$\V$-facial cluster}.
Similarly, we define the \emph{graph of $\W$-faces} and  a \textit{$\W$-facial cluster}.
In a $\V$-facial cluster, every weak reducing pair gives the common $\W$-disk, and vise versa.
\end{definition}

If we consider an unstabilized, genus three Heegaard splitting of an irreducible manifold, then we get the following lemmas.

\begin{lemma}[J. Kim, \cite{JungsooKim2012}]\label{lemma-2-14}
Suppose that $M$ is an irreducible $3$-manifold and $(\V,\W;F)$ is a genus three, unstabilized Heegaard splitting of $M$.
If there are two $\V$-faces $f_1$ represented by $\{V_0,V_1,W\}$ and $f_2$ represented by $\{V_1, V_2, W\}$ sharing a weak reducing pair $(V_1,W)$, then $\partial V_1$ is non-separating, and $\partial V_0$, $\partial V_2$ are separating in $F$.
Therefore, there is a unique weak reducing pair in a $\V$-facial cluster which can belong to two or more faces in the $\V$-facial cluster.
\end{lemma}

\begin{definition}[J. Kim, \cite{JungsooKim2012}]\label{definition-2-15}
By Lemma \ref{lemma-2-14}, there is a unique weak reducing pair in a $\V$-facial cluster belonging to two or more faces in the cluster.
We call it the \textit{center} of a $\V$-facial cluster.
We call the other weak reducing pairs \textit{hands} of a $\V$-facial cluster.
See Figure \ref{figure-2}.
Note that if a $\V$-face is represented by two weak reducing pairs, then one is the center and the other is a hand.
Lemma \ref{lemma-2-14} means that the $\V$-disk in the center of a $\V$-facial cluster is non-separating, and those from hands are all separating.
Moreover, Lemma \ref{lemma-2-9} implies that (i) the $\V$-disk in a hand of a $\V$-facial cluster is a band-sum of two parallel copies of that of the center of the $\V$-facial cluster and (ii) the $\V$-disk of a hand of a $\V$-facial cluster determines that of the center of the $\V$-facial cluster by the uniqueness of the meridian disk of the solid torus which the $\V$-disk of the hand cuts off from $\V$.
\end{definition}
	
Note that every $\V$ - or $\W$- facial cluster is contractible (see Figure \ref{figure-2}).

\begin{figure}
\includegraphics[width=4.5cm]{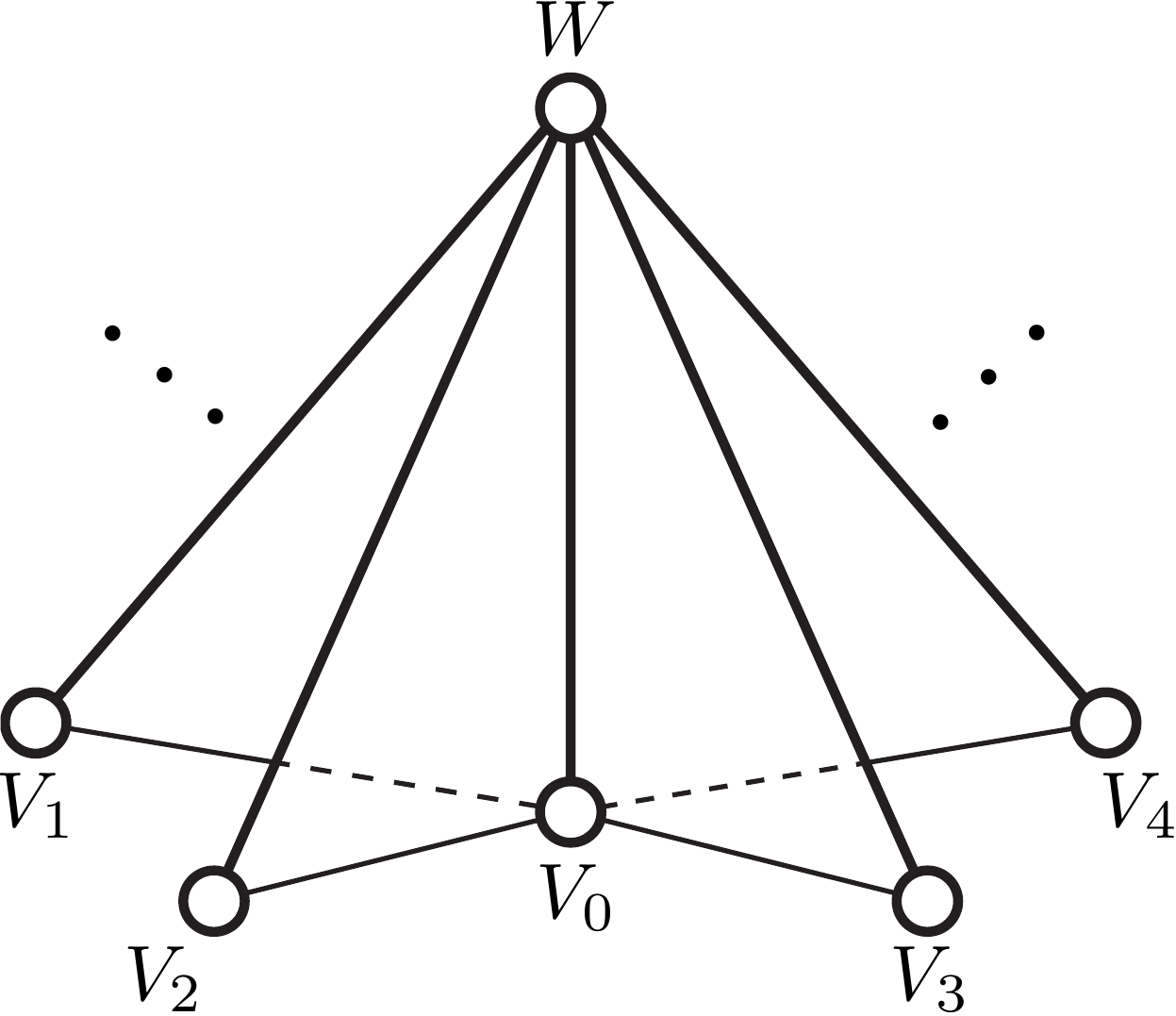}
\caption{An example of a $\V$-facial cluster in $\DVW(F)$. $(V_0, W)$ is the center and the other weak reducing pairs are hands. \label{figure-2}}
\end{figure}

\begin{lemma}[J. Kim, \cite{JungsooKim2012}]\label{lemma-2-16} 
Assume $M$ and $F$ as in Lemma \ref{lemma-2-14}.
Every $\V$-face belongs to some $\V$-facial cluster. 
Moreover, every $\V$-facial cluster has infinitely many hands.
\end{lemma}

The following lemma means that the isotopy class of generalized Heegaard splitting obtained by weak reduction along a weak reducing pair does not depend on the choice of the weak reducing pair if the weak reducing pair varies in a fixed $\V$- or $\W$-facial cluster.

\begin{lemma}\label{lemma-2-17}
Assume $M$ and $F$ as in Lemma \ref{lemma-2-14}.
Every weak reducing pair in a $\V$-face gives the same generalized Heegaard splitting after weak reduction up to isotopy.
Therefore, every weak reducing pair in a $\V$-facial cluster gives the same generalized Heegaard splitting after weak reduction up to isotopy.
Moreover, the embedding of the thick level contained in $\V$ or $\W$ does not vary in the relevant compression body up to isotopy. 
\end{lemma}

\begin{proof}
Let $(V,W)$ be the center of a $\V$-facial cluster and $(V',W)$ be a hand of the $\V$-facial cluster.
Here, $V$ is non-separating and $V'$ is separating in $\V$ by Lemma \ref{lemma-2-14}.
Let $\mathbf{H}$ and $\mathbf{H}'$ be the generalized Heegaard splittings obtained by weak reductions along $(V,W)$ and $(V',W)$ from $(\V,\W;F)$ respectively.
It is sufficient to show that $\mathbf{H}$ and $\mathbf{H'}$ are the same up to isotopy.\\

\Claim{Both $\overline{\operatorname{Thin}}(\mathbf{H})$ and $\overline{\operatorname{Thin}}(\mathbf{H'})$ consist of one component.}

\begin{proofc}
Suppose that $\overline{\operatorname{Thin}}(\mathbf{H})$ or $\overline{\operatorname{Thin}}(\mathbf{H'})$ does not consist of one component.

We claim that each component of the inner thin level must have scars of both disks of the weak reducing pair.
Let us consider an arbitrary weak reducing pair $(V,W)$.
Then $\partial W$ must belong to the genus two component of $F-\partial V$ by Lemma \ref{lemma-2-8} and vise versa.
Hence, we get the follows in the pseudo-GHS obtained by the pre-weak reduction along $(V,W)$.
\begin{enumerate}
\item $F_{VW}$ has a component not having scars of both $V$ and $W$ if and only if at least one of $V$ and $W$ is separating.
\item If one of $V$ and $W$ is separating, say $V$, then we can find a product region in the pseudo-GHS cobounded by the isotoped genus one component of $F_V$ into the interior of $\V$ and the torus component of $F_{VW}$ having only the scar of $V$.
\end{enumerate}
Therefore, every component of $F_{VW}$ not having scars of both $V$ and $W$ disappears after cleaning (see Figure \ref{figure-3} or Figure \ref{figure-wrc}).
\begin{figure}
\includegraphics[width=12cm]{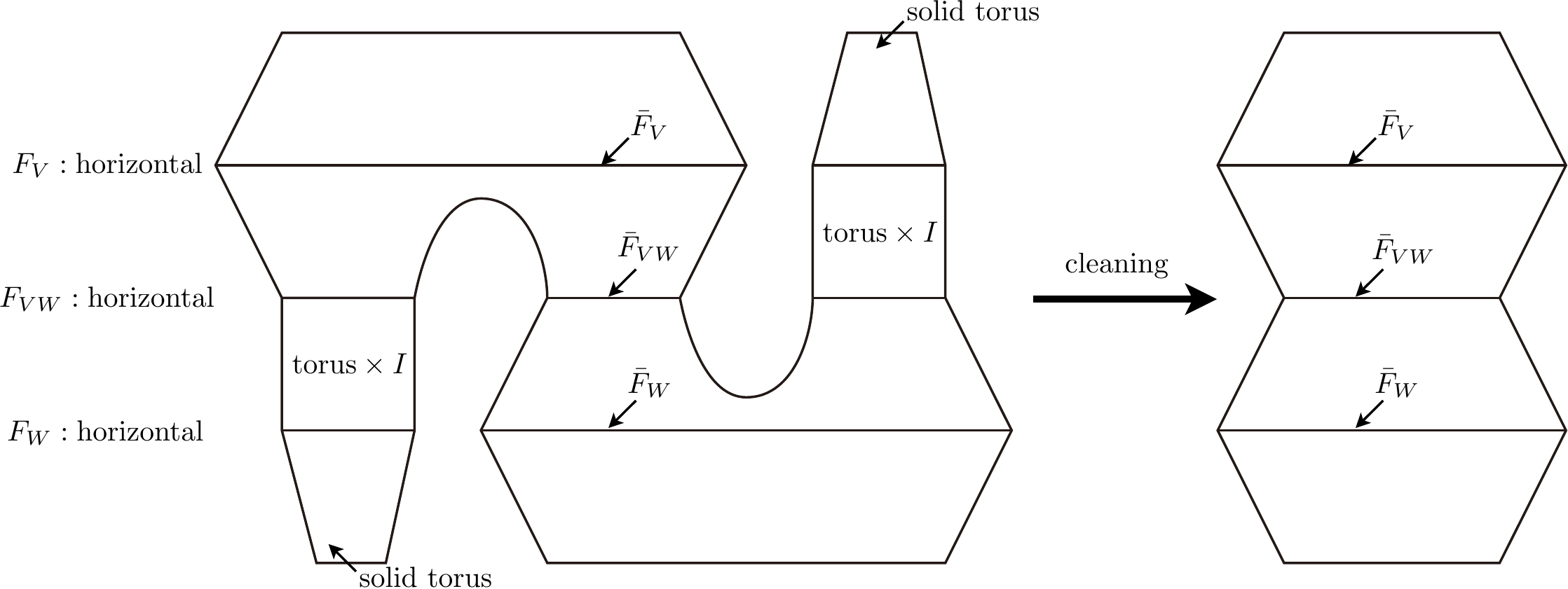}
\caption{The inner thin level comes from the component of $F_{VW}$ having scars of both $V$ and $W$. \label{figure-3}}
\end{figure}

Hence, if we try all possible weak reductions by considering Lemma \ref{lemma-2-8} (see Appendix \ref{appendix}), then the only case for disconnected inner thin level is when both disks are non-separating but the union of boundaries of them is separating in $F$, i.e. the inner thin level consists of two tori (see Figure \ref{figure-4}).
This means that $\overline{\operatorname{Thin}}(\mathbf{H'})$ must consist of only one component since $V'$ is separating.
Hence, $\overline{\operatorname{Thin}}(\mathbf{H})$ is disconnected, i.e. $W$ is non-separating and $\partial V\cup\partial W$ is separating in $F$.
But Lemma \ref{lemma-2-9} forces $V'$ to be a band-sum of two parallel copies of $V$.
Here, $V'$ must intersect $W$ otherwise we can find an arc in $F$ realizing the band-sum but missing $\partial W$, i.e. $\partial V\cup\partial W$ is non-separating, violating the assumption that $\partial V\cup\partial W$ is separating in $F$.
But this violates the assumption that $(V',W)$ is a weak reducing pair.

This completes the proof of Claim.
\end{proofc}

If we consider a pseudo-GHS $\mathbf{G}'$ obtained by pre-weak reduction along a weak reducing pair $(D,E)$ and do cleaning, then the genus one component of $F_D$ in $T(\mathbf{G}')$ disappears after cleaning when $D$ is separating since the region between $F_{DE}$ and the genus one component of $F_D$ in $T(\mathbf{G}')$ is $(\text{torus})\times I$  as we have seen in the proof of Claim and the symmetric argument also holds for the genus one component of $F_E$ in $T(\mathbf{G}')$ when $E$ is separating.

Hence, if we consider $\operatorname{Thick}(\mathbf{H})$ and $\overline{\operatorname{Thin}}(\mathbf{H})$, then we get
$$\operatorname{Thick}(\mathbf{H})=\{F_V,\bar{F}_{W}\}, \quad \overline{\operatorname{Thin}}(\mathbf{H})=\{\bar{F}_{VW}\},$$
where $\bar{F}_{W}$ comes from the genus two component of $F_W$ and $\bar{F}_{VW}$ is the component of $F_{VW}$ having scars of both $V$ and $W$ (if there is no confusion, then we will use the terms $F_V$ or $\bar{F}_V$ as the one isotoped into the interior of $\V$ for the cases of thick levels).
Similarly, if we consider $\operatorname{Thick}(\mathbf{H'})$ and $\operatorname{Thin}(\mathbf{H'})$, then we get
$$\operatorname{Thick}(\mathbf{H'})=\{\bar{F}_{V'},\bar{F}_{W}\},\quad \overline{\operatorname{Thin}}(\mathbf{H'})=\{\bar{F}_{V'W}\},$$
where $\bar{F}_{V'}$ comes from the genus two component of $F_{V'}$ and $\bar{F}_{V'W}$ is the component of $F_{V'W}$ having scars of both $V'$ and $W$.
Here, $V'$ must cut off a solid torus $\V'$ from $\V$ and $V$ is a meridian disk of $\V'$ by Lemma \ref{lemma-2-9}, i.e. $F_V$ is isotopic to $\bar{F}_{V'}$ in $\V$, so is in $M$. 
Moreover, $\bar{F}_{VW}$ is isotopic to $\bar{F}_{V'W}$ similarly since $\partial W$ must belong to the genus two component of $F-\partial V'$ by Lemma \ref{lemma-2-8}.
Hence, $\operatorname{Thick}(\mathbf{H})=\operatorname{Thick}(\mathbf{H'})$ and $\operatorname{Thin}(\mathbf{H})=\operatorname{Thin}(\mathbf{H'})$ up to isotopy.
This completes the proof.
\end{proof}

%Note that the proof of Lemma \ref{lemma-2-17} implies that we can characterize the generalized Heegaard splittings obtained by weak reductions from $(\V,\W;F)$ into the following five types with respect to the shape of the two compression bodies sharing the inner thin level, i.e.  (i) (a) of Figure \ref{figure-5}, (ii) (b) of Figure \ref{figure-5}, (iii) the symmetric one of (b) of Figure \ref{figure-5}, (iv) (c) of Figure \ref{figure-5} and (v) Figure \ref{figure-4}.

The next lemma gives an upper bound for the dimension of $\DVW(F)$ and restricts the shape of a $3$-simplex in $\DVW(F)$.

\begin{lemma}[J. Kim, Proposition 2.10 of \cite{JungsooKim2013}]\label{lemma-2-18}
Assume $M$ and $F$ as in Lemma \ref{lemma-2-14}.
Then $\operatorname{dim}(\DVW(F))\leq 3$.
Moreover, if $\operatorname{dim}(\DVW(F))=3$, then every $3$-simplex in $\DVW(F)$ must have the form $\{V_1, V_2, W_1, W_2\}$, where $V_1, V_2\subset \V$ and $W_1,W_2\subset \W$.
Indeed, $V_1$ ($W_1$ resp.) is non-separating in $\V$ (in $\W$ resp.) and $V_2$ ($W_2$ resp.) is a band-sum of two parallel copies of $V_1$ in $\V$ ($W_1$ in $\W$ resp.).
\end{lemma}

Note that the third statement of Lemma \ref{lemma-2-18} is obtained by applying Lemma \ref{lemma-2-9} to the $\V$-face $\{V_1,V_2,W_1\}$ and the $\W$-face $\{V_2,W_1,W_2\}$.

The next lemma gives natural, but important observations for genus $g\geq 2$ compression bodies containing a minus boundary component of genus $g-1$.
The proof comes from a standard outermost disk argument for the intersection of two compressing disks in $\V$ when we consider the uniqueness of the wanted disk and we can find a rigorous proof in Lemma 3.3 of \cite{IdoJangKobayashi2014}.

\begin{lemma}\label{lemma-2-19}
Let $\V$ be a genus $g\geq 2$ compression body with $\partial_-\V$ containing a genus $g-1$ surface.
Then the follows hold.
\begin{enumerate}
\item If $\partial_-\V$ is connected (i.e. $\partial_-\V$ consists of a genus $g-1$ surface), then there is a unique non-separating disk in $\V$ up to isotopy.\label{lemma-2-19-1}
\item If $\partial_-\V$ is disconnected (i.e. $\partial_-\V$ consists of a genus $g-1$ surface and a torus), then there is a unique compressing disk in $\V$ up to isotopy which is separating in $\V$.\label{lemma-2-19-2}
\end{enumerate}
\end{lemma}

\section{The proof of Theorem \ref{theorem-main}}

First, we introduce the following lemma.

\begin{lemma}\label{lemma-3-1}
Suppose that $M$ is an orientable, irreducible $3$-manifold and $(\V,\W;F)$ is a genus three, unstabilized Heegaard splitting of $M$.
If there exist two weak reducing pairs such that the generalized Heegaard splittings obtained by weak reductions along these weak reducing pairs are not isotopic in $M$, then $F$ is critical.
\end{lemma}

Before proving Lemma \ref{lemma-3-1}, we introduce the next definition.

\begin{definition}[Bachman, Definition 8.3 of \cite{Bachman2008}]\label{definition-distance-wrp}
Suppose $F$ is a Heegaard surface in a $3$-manifold.
Let $(V_i, W_i)$ be a weak reducing pair for $F$ for $i= 0, 1$.
Then we define the \textit{distance} between $(V_0, W_0)$ and  $(V_1, W_1)$ to be the smallest $n$ such that there is a sequence $\{D_j\}_{j=0}^{n+1}$ where
\begin{enumerate}
\item $\{D_0, D_1\}=\{V_0,W_0\}$,
\item $\{D_n, D_{n+1}\}=\{V_1, W_1\}$,
\item for all $j$ the pair $\{D_j,D_{j+1}\}$ gives a weak reducing pair for $F$,
\item for $1\leq j\leq n$, $D_{j-1}$ is disjoint from, or equal to, $D_{j+1}$.
\end{enumerate}
If there is no such sequence, then we define the distance to be $\infty$.
\end{definition}

\begin{lemma}[Bachman, Lemma 8.5 of \cite{Bachman2008}]\label{lemma-3-3}
Suppose $F$ is a Heegaard surface in a $3$-manifold. 
If there are weak reducing pairs $(V,W)$ and $(V',W')$ for $F$ such that the distance between $(V,W)$ and $(V',W')$ is $\infty$, then $F$ is critical.
\end{lemma}

\begin{proofN}{Lemma \ref{lemma-3-1}}
Suppose that there are two weak reducing pairs $(V, W)$ and $(V', W')$ for $(\V,\W;F)$ such that the generalized Heegaard splittings $\mathbf{H}$ and $\mathbf{H}'$ obtained by weak reductions along these weak reducing pairs are not isotopic in $M$.

If the distance between $(V,W)$ and $(V',W')$ is $\infty$, then $F$ is critical by Lemma \ref{lemma-3-3}.

Assume that the distance between $(V,W)$ and $(V',W')$ is $k<\infty$.
If $k\leq 1$, then either $(V,W)=(V',W')$ or they are contained in a $\V$- or $\W$-face.
This leads to a contradiction by Lemma \ref{lemma-2-17}.
Hence, assume that $k\geq 2$.
Let $\{D_0, D_1, \cdots, D_k, D_{k+1}\}$ be the sequence of compressing disks of $F$ realizing the distance between $(V,W)$ and $(V',W')$, where $\{D_0,D_1\}=\{V,W\}$ and $\{D_k,D_{k+1}\}=\{V',W'\}$.
By reading the above sequence from the left to the right, we get a sequence of $\V$- and $\W$-faces $\Delta_0$, $\cdots$, $\Delta_n$  such that
\begin{enumerate}
\item  $(V, W)\subset \Delta_0$ and $(V', W')\subset \Delta_n$,
\item $\Delta_{i-1}$ shares a weak reducing pair with $\Delta_{i}$ for $i=1, \cdots, n$.
\end{enumerate}
If we consider the generalized Heegaard splitting corresponding to $\Delta_i$ inductively from $i=0$ to $n$ by using Lemma \ref{lemma-2-17} and the assumption that $\Delta_{i-1}$ shares a weak reducing pair with $\Delta_i$, then we can see that $\mathbf{H}$ and $\mathbf{H}'$ are isotopic, violating the assumption.
Hence, the distance between $(V,W)$ and $(V',W')$ cannot be finite.
This completes the proof.
\end{proofN}

Let $(\V,\W;F)$ be a weakly reducible, unstabilized, genus three Heegaard splitting of an irreducible $3$-manifold $M$.
By considering Lemma \ref{lemma-3-1}, we assume that the generalized Heegaard splitting obtained by weak reduction from $(\V,\W;F)$ is unique up to isotopy.
If we consider the generalized Heegaard splitting obtained by weak reduction along $(V,W)$ from $(\V,\W;F)$, then the inner thin level would consist of a torus or two tori, where the latter case holds only when both disks of the weak reducing pair are non-separating but the union of the boundaries is separating in $F$ as we have checked in the proof of Lemma \ref{lemma-2-17} (see Figure \ref{figure-4}).
\begin{figure}
\includegraphics[width=4cm]{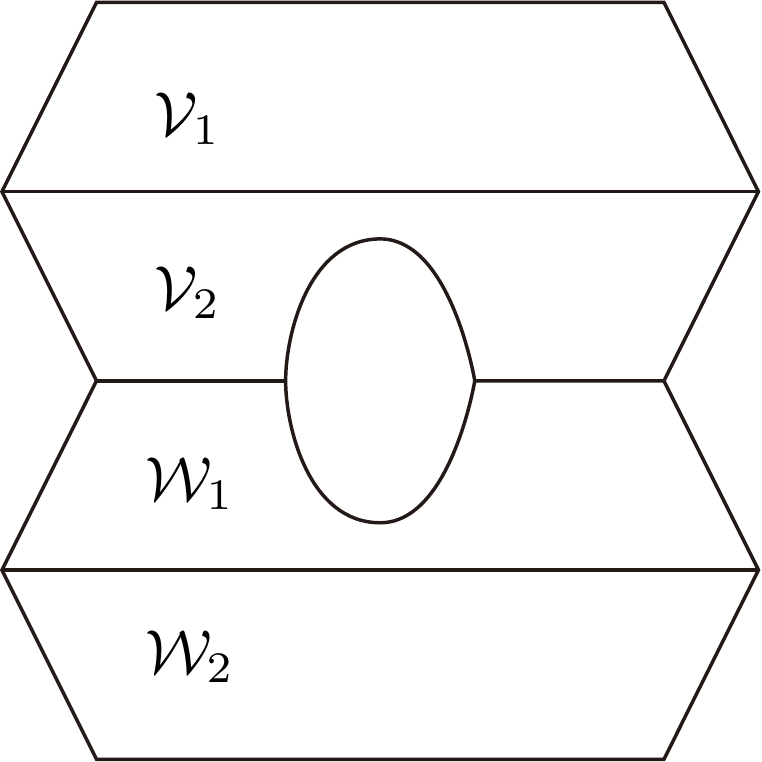}
\caption{The inner thin level consists of two tori. \label{figure-4}}
\end{figure}
The thick levels would come from the genus two components of $F_V$ and $F_W$, say $\bar{F}_V$ and $\bar{F}_W$ as we have checked in the proof of Lemma \ref{lemma-2-17}.
But even though we assumed that the generalized Heegaard splitting obtained by weak reduction from $(\V,\W;F)$ is unique up to isotopy, it is not clear that both $\bar{F}_V$ and $\bar{F}_W$ should be unique up to isotopy in $\V$ and $\W$ respectively, not in the entire $M$.
Indeed, we can imagine an ambient isotopy $f_t$ defined in $M$ such that $f_0$ is the identity map in $M$, $f_1(\bar{F}_V)\cap \W= \emptyset$, $\bar{F}_V$ is not isotopic to $f_1(\bar{F}_V)$ in $\V$, and $f_t(\bar{F}_V)\cap \W\neq \emptyset$ for some $t$.
If the isotoped generalized Heegaard splitting itself  is also that obtained by weak reduction from $(\V,\W;F)$, then there would be another weak reducing pair $(V',W')$ for $(\V,\W;F)$ such that $f_1(\bar{F}_V)$ is isotopic to the genus two component of $F_{V'}$ in $\V$ but $V$ is not isotopic to $V'$ in $\V$.
Hence, we need the following lemma.

\begin{lemma}\label{lemma-new}
Assume $M$ and $F$ as in Lemma \ref{lemma-3-1}.
Suppose that there are two generalized Heegaard splittings $\mathbf{H}_1$ and $\mathbf{H}_2$ obtained by weak reductions along $(V_1,W_1)$ and $(V_2,W_2)$ from $(\V,\W;F)$ respectively such that the thick levels of $\mathbf{H}_1$ and $\mathbf{H}_2$ embedded in the interior of $\V$ are non-isotopic in $\V$. (It may be possible that $\mathbf{H}_1$ is the same as $\mathbf{H}_2$ in $M$ up to isotopy.)
Then $F$ is critical.
\end{lemma}

\begin{proof}
Say $\mathbf{H}_i=(\V_1^i,\V_2^i;\bar{F}_{V_i})\cup(\W_1^i,\W_2^i;\bar{F}_{W_i})$, where $\partial_-\V_2^i\cap\partial_-\W_1^i\neq\emptyset$ and let $\V_i'$ be the solid between $\bar{F}_{V_i}$ and $\partial_+\V$ in $\V$ for $i=1,2$.
By construction, $V_i$ is a compressing disk of $\V_i'$.
Hence, it is clear that $\V_i'$ is a genus three compression body, where either $\partial_-\V_i'$ consists of (i) a genus two surface if $\partial_-\V_i'$ is connected or (ii) a torus and a genus two surface if $\partial_-\V_i'$ is disconnected for $i=1,2$.\\

\Claim{If one of $\partial_-\V_1'$ and $\partial_-\V_2'$ is connected and the other is disconnected, then $F$ is critical.}

\begin{proofc}
Suppose that $\partial_-\V_1'$ is connected but $\partial_-\V_2'$ is disconnected.

If we consider $V_1$, then either $V_1$ is non-separating in $\V$ or $V_1$ cuts off a solid torus from $\V$ since $\partial_-\V_1'$ is connected.
If $\partial_-\W_1^1$ has a component not belonging to the inner thin level, then this component cannot come from $\partial_-\W$ since $\W_1^1\cap \W$ is the region  in $\W$ between the thick level $\bar{F}_{W_1}$ and the genus two component of $F_{W_1}$ which is homeomorphic to $\bar{F}_{W_1}\times I$ (see Figure \ref{figure-wrc}).
\begin{figure}
\includegraphics[width=13cm]{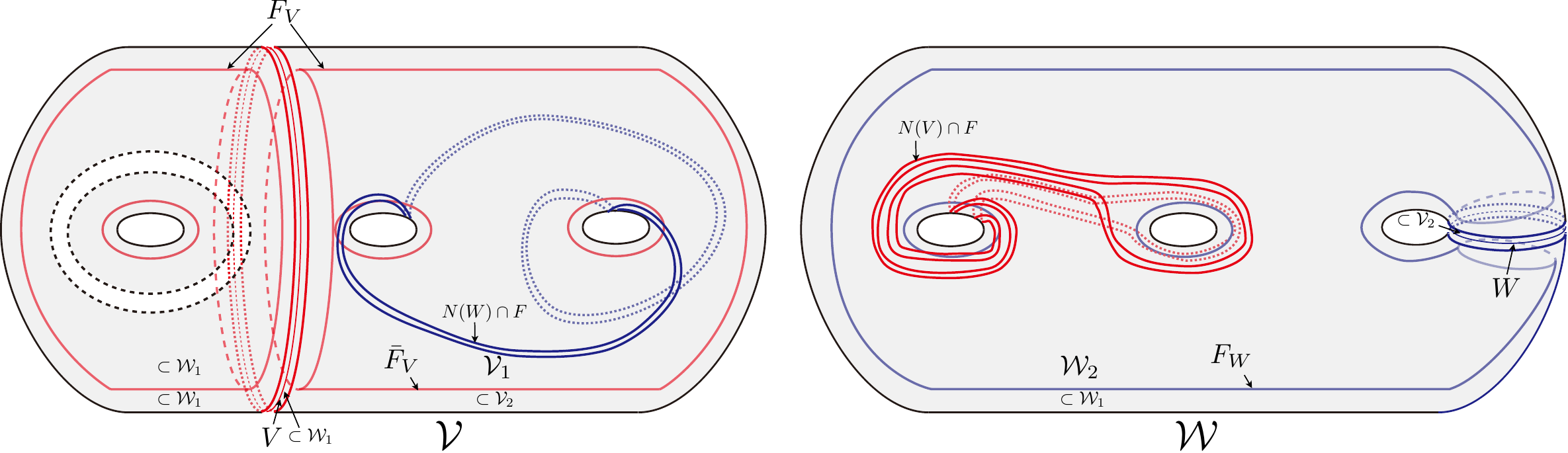}
\caption{$\W_1\cap \W \cong \bar{F}_{W}\times I$ and $\V_2\cap \V \cong \bar{F}_{V}\times I$ \label{figure-wrc}}
\end{figure}
Hence, this component comes from $\partial_-\V$ after cleaning (see Figure (b) or (c) of Figure \ref{figure-5}, where $\partial_-\V$ comes from the top horizontal line) and therefore $V_1$ must cut off $(\text{torus})\times I$ from $\V$, violating the assumption.
Therefore, $\partial_-\W_1^1$ is the inner thin level itself, i.e. $\partial_-\W_1^1\cap\partial_-\V=\emptyset$. 

Now we consider $V_2$.
Since $\V_2'$ is a genus three compression body such that $\partial_-\V_2'$ consists of a torus and a genus two surface, $V_2$ must cut off $(\text{torus})\times I$ from $\V$ by Lemma \ref{lemma-2-19}.
Hence, the region cut by a copy of $V_2$ from $\V$ which is homeomorphic to $(\text{torus})\times I$ would be attached to the product region in $\W$ between the thick level $\bar{F}_{W_2}$ and the genus two component of $F_{W_2}$ to complete $\W_1^2$ (see Figure \ref{figure-wrc} or (b) or (c) of Figure \ref{figure-5}), i.e. $\partial_-\W_1^2\cap\partial_-\V\neq\emptyset$.

This means that if $\mathbf{H}_1$ is isotopic to $\mathbf{H}_2$, then the isotopy cannot take $\W_1^1$ into $\W_1^2$, i.e. it takes $\V_2^1$ into $\W_1^2$ since the isotopy takes the inner thin level of $\mathbf{H}_1$ into that of $\mathbf{H}_2$.
But $\partial_-\V_2^1\cap\partial M\subset \partial_-\W$ if it is nonempty (for example, we can refer to (c) of Figure \ref{figure-5}).
Since the isotopy cannot change $\partial_-\V_2^1\cap \partial M$, we conclude that $\mathbf{H}_1$ cannot be isotopic to $\mathbf{H}_2$ in $M$.
Therefore, $F$ is critical by Lemma \ref{lemma-3-1}.
This completes the proof of Claim.
\end{proofc}

By Claim, we can assume that both $\partial_-\V_1'$ and $\partial_-\V_2'$ are connected or both $\partial_-\V_1'$ and $\partial_-\V_2'$ are disconnected.
If $V_1$ is isotopic to $V_2$ in $\V$, then $\bar{F}_{V_1}$ would be isotopic to $\bar{F}_{V_2}$ in $\V$ in any case, violating the assumption.
Hence, $V_1$ is not isotopic to $V_2$ in $\V$.

Suppose that both $\partial_-\V_1'$ and $\partial_-\V_2'$ are connected.
If $V_i$ is separating in $\V_i'$, then it must cut off a solid torus from $\V_i'$ since $\V_i'$ is a genus three compression body such that $\partial_-\V_i'$ consists of a genus two surface.
Hence, we can take a meridian disk $V_i'$ of the solid torus which $V_i$ cuts off from $\V_i'$ so that it would miss $V_i$.
Moreover, $V_i'\cap W_i=\emptyset$ by Lemma \ref{lemma-2-8}.
That is, we get the $\V$-face $\{V_i',V_i,W_i\}$. 
Hence, we can assume that $V_i$ is non-separating without changing the isotopy class of $\mathbf{H}_i$ and the embedding of $\bar{F}_{V_i}$ in $\V$ up to isotopy by Lemma \ref{lemma-2-17}.
Since $V_1$ is not isotopic to $V_2$ in $\V$ and both disks are non-separating in $\V$, $F$ is critical by Theorem 1.1 of \cite{JungsooKim2013}.

Hence, we can assume that both $\partial_-\V_1'$ and $\partial_-\V_2'$ are disconnected, i.e. each $V_i$ cuts off $(\text{torus})\times I$ from $\V_i'$, so also does in $\V$, for $i=1,2$.
We claim that the distance defined in Definition \ref{definition-distance-wrp} between $(V_1,W_1)$ and $(V_2,W_2)$ is $\infty$.

For the sake of contradiction, assume that the distance is finite.
Then we get a sequence of $\V$- and $\W$-faces $\Delta_0$, $\cdots$, $\Delta_n$   such that
\begin{enumerate}
\item $(V_1, W_1)\subset \Delta_0$ and $(V_2, W_2)\subset \Delta_n$,
\item $\Delta_{i-1}$ shares a weak reducing pair with $\Delta_{i}$ for $i=1, \cdots, n$, 
\end{enumerate}
similarly as in the proof of Lemma \ref{lemma-3-1}.
Since $V_1$ is not isotopic to $V_2$ in $\V$, there must be a $\V$-face among $\Delta_0$, $\cdots$, $\Delta_n$.
Let $\Delta_k$ be the first $\V$-face, i.e. it contains $V_1$.
Here, every $\V$-face contains a non-separating $\V$-disk and the boundary of it must be a non-separating loop in the punctured torus which the boundary of  the separating $\V$-disk in the $\V$-face cuts off from $F$ by  Lemma \ref{lemma-2-9}.
But the condition that $V_1$ cuts off $(\text{torus})\times I$ from $\V$ means that there cannot be such a non-separating $\V$-disk in $\Delta_k$, leading to a contradiction.

Hence, the distance between $(V_1,W_1)$ and $(V_2,W_2)$ is $\infty$, i.e. $F$ is critical by Lemma \ref{lemma-3-3}.

This completes the proof of Lemma \ref{lemma-new}.
\end{proof}

Now we introduce the next lemma dealing with the case when the inner thin level consists of a torus.

\begin{lemma}\label{lemma-3-4}
Let $(\V,\W;F)$ be a weakly reducible, unstabilized, genus three Heegaard splitting in an orientable, irreducible $3$-manifold.
If every weak reducing pair of $F$ gives the same generalized Heegaard splitting obtained by weak reduction up to isotopy such that the inner thin level consists of a torus and the embedding of each thick level in the relevant compression body is also unique up to isotopy, then $\D(F)$ is contractible.
\end{lemma}

\begin{proof}
Let $(V, W)$ be a weak reducing pair for $(\V,\W;F)$ and $T$ the inner thin level of the generalized Heegaard splitting obtained by weak reduction.

If one of $V$ and $W$, say $V$, cuts off a solid torus from $\V$, then $\partial W$ cannot belong to the once-punctured torus that $\partial V$ cuts off from $F$ by Lemma \ref{lemma-2-8}.
Hence, we can take a non-separating disk $V'$ from the solid torus and we can assume that $V' \cap (V\cup W)=\emptyset$.
Of course, the generalized Heegaard splitting obtained by weak reduction along $(V,W)$ is the same as the one obtained by weak reduction along $(V',W)$ and the embeddings of thick levels in each compression body are the same  up to isotopy  by Lemma \ref{lemma-2-17}.
Hence, without loss of generality, there are three types of the generalized Heegaard splittings obtained by weak reductions as follows, where these cases coming from the shape of the two compression bodies sharing the inner thin level. 
\begin{enumerate}[(a)]
\item We can assume that $V$ and $W$ are non-separating in $\V$ and $\W$ respectively, i.e. the minus boundaries of the two compression bodies are connected (see (a) of Figure \ref{figure-5}), \label{as-a}
\item $V$ cuts off $(\text{torus})\times I$ from $\V$ and we can assume that $W$ is non-separating in $\W$, i.e. the minus boundary of one compression body is connected but that of the other is disconnected (see  (b) of Figure \ref{figure-5}), or \label{as-b}
\item each of $V$ and $W$ cuts off $(\text{torus})\times I$ from $\V$ or $\W$ respectively, i.e. the minus boundaries of the two compression bodies are disconnected (see (c) of Figure \ref{figure-5}),\label{as-c}
\end{enumerate}
where these three cases are mutually exclusive by the assumption that the generalized Heegaard splitting obtained by weak reduction is unique up to isotopy.
Note that there is the symmetric case for (\ref{as-b}) when $W$ cuts off $(\text{torus})\times I$ from $\W$, but the shape of the generalized Heegaard splitting is just the one obtained by turning the figure upside down.\\

\begin{figure}
\includegraphics[width=13cm]{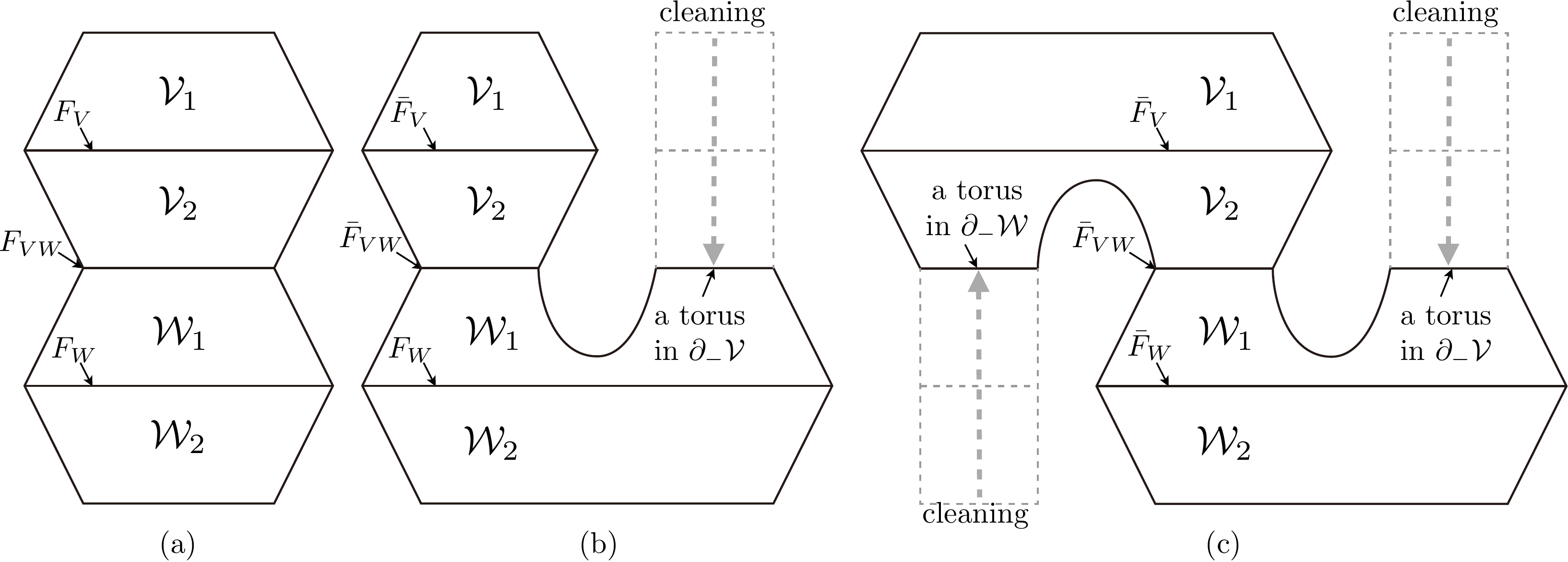}
\caption{the GHSs for the three cases \label{figure-5}}
\end{figure}

\Case{(\ref{as-a})} We can assume that $V$ and $W$ are non-separating in $\V$ and $\W$ respectively.

In this case, $\partial V\cup\partial W$ is non-separating in $F$ since the inner thin level consists of a torus.

Let us consider the generalized Heegaard splitting obtained by weak reduction along $(V,W)$. 
Here, we can see that the pre-weak reduction along $(V,W)$ is exactly the same as the weak reduction along $(V,W)$.
Hence, it consists of two splittings $(\V_1, \V_2;F_V)$ and $(\W_1, \W_2;F_W)$ such that $\partial_- \V_2=\partial_-\W_1=T$.
Let $\V'$ ($\W'$ resp.) be the solid between $\partial_+\V$ and $F_V$ in $\V$ ($\partial_+\W$ and $F_W$ in $\W$ resp.).
Then we can see that $V$ and $W$ are non-separating compressing disks of $\V'$ and $\W'$ obviously.
Since $\V'$ and $\W'$ are genus three compression bodies with minus boundary consisting of a genus two surface, $V$ and $W$ are uniquely determined in $\V'$ and $\W'$ respectively up to isotopy.
Hence, we get the induced isotopy classes of $V$ and $W$ in $\V$ and $\W$ from the isotopy classes of $V$ and $W$ in $\V'$ and $\W'$.
Moreover, the uniqueness of the isotopy classes of the embeddings of the thick levels of the generalized Heegaard splitting obtained by weak reduction from $(\V,\W;F)$ in the relevant compression bodies forces the choice of the induced isotopy classes of $V$ and $W$ in $\V$ and $\W$ to be unique.
This means that we can consider the weak reducing pair $(V,W)$ as the fixed one $(\bar{V},\bar{W})$ even though we've chosen an arbitrary weak reducing pair consisting of non-separating disks.

Let us consider an arbitrary weak reducing pair $(V^\ast, W^\ast)$ for $(\V,\W;F)$.

If both $V^\ast$ and $W^\ast$ are non-separating, then $(V^\ast, W^\ast)$ must be $(\bar{V},\bar{W})$ by the previous argument and therefore $\partial V^\ast\cup\partial W^\ast$ is non-separating in $F$.
This means that we can take a band-sum of two parallel copies of $V^\ast$ in $\V$, say $V'$, and a band-sum of two parallel copies of $W^\ast$ in $\W$, say $W'$, such that $\{V',V^\ast=\bar{V},W^\ast=\bar{W},W'\}$ forms a $3$-simplex.

If exactly one of $V^\ast$ and $W^\ast$ is non-separating, say $V^\ast$, then $W^\ast$ must cut off a solid torus from $\W$ (otherwise, $W^\ast$ cuts off $(\text{torus})\times I$ from $\W$ and therefore the generalized Heegaard splitting obtained by weak reduction along $(V^\ast,W^\ast)$ would be the symmetric case of (\ref{as-b}) or (\ref{as-c}) of Figure \ref{figure-5}, violating the uniqueness of the generalized Heegaard splitting obtained by weak reduction up to isotopy). 
If we choose the meridian disk of the solid torus which $W^\ast$ cuts off from $\W$ so that it would miss $W^\ast$, say $\tilde{W}$, then $\tilde{W}$ also misses $V^\ast$ by Lemma \ref{lemma-2-8}.
Hence, $\partial V^\ast$ belongs to the genus two component of $F-\partial W^\ast$ and therefore we can take a band-sum of two parallel copies of $V^\ast$ in $\V$ which misses $W^\ast$, say $V'$.
That is, $\{V',V^\ast,\tilde{W},W^\ast\}$ forms a $3$-simplex in $\DVW(F)$ and the weak reducing pair $(V^\ast,\tilde{W})$ consisting of non-separating disks must be $(\bar{V},\bar{W})$ by the previous argument.

If both $V^\ast$ and $W^\ast$ are separating, then we can find a weak reducing pair $(\tilde{V}, \tilde{W})$ consisting of non-separating disks, where $\tilde{V}$ ($\tilde{W}$ resp.) comes from the meridian disk of the solid torus which $V^\ast$ ($W^{\ast}$ resp.) cuts off from the corresponding compression body.
Here, we can assume that $\tilde{V}\cup\tilde{W}$ misses $V^\ast\cup W^\ast$, i.e. $\{V^\ast,\tilde{V},\tilde{W}, W^\ast\}$ forms a $3$-simplex in $\DVW(F)$.
If we apply the previous argument to $(\tilde{V},\tilde{W})$, then $(\tilde{V},\tilde{W})$ would be $(\bar{V}, \bar{W})$.

This means that an arbitrary weak reducing pair $(V^\ast, W^\ast)$ of $(\V,\W;F)$ belongs to some $3$-simplex of the form $\Sigma_{V' W'}=\{V',\bar{V},\bar{W}, W'\}$ in $\DVW(F)$ containing the fixed $1$-simplex $\{\bar{V},\bar{W}\}$, where $V'\subset \V$ and $W'\subset \W$ are band-sums of two parallel copies of $\bar{V}$ and $\bar{W}$ in $\V$ and $\W$ respectively by Lemma \ref{lemma-2-18}.\\

\ClaimN{A}{
$\DVW(F)=\bigcup_{V', W'} \Sigma_{V' W'}$ for all possible $V'$ and $W'$.}

\begin{proofN}{Claim A}
Since $\bigcup_{V', W'} \Sigma_{V' W'}\subset \DVW(F)$ is obvious, we will prove that every simplex of $\DVW(F)$ belongs to some $\Sigma_{V' W'}$.

By definition of $\DVW(F)$ and the assumption that every weak reducing pair belongs to some $\Sigma_{V' W'}$, we don't need to consider  vertices or $1$-simplices in $\DVW(F)$.

If there is a $2$-simplex $\Delta$ in $\DVW(F)$, then it must be a $\V$-face or a  $\W$-face.
Otherwise, we can assume that $\Delta\subset \DV(F)$ without loss of generality and there must be a vertex in $\DW(F)$ such that $\Delta$ forms a $3$-simplex in $\DVW(F)$ together with this vertex by the definition of $\DVW(F)$.
Hence, three vertices of the $3$-simplex come from $\DV(F)$, violating Lemma \ref{lemma-2-18}.
Without loss of generality, suppose that $\Delta$ is a $\V$-face.
That is, there is a non-separating $\V$-disk and a separating $\V$-disk in $\Delta$ by Lemma \ref{lemma-2-9}.
If the $\W$-disk of $\Delta$ is separating, then it cannot cut off $(\text{torus})\times I$ from $\W$ by the uniqueness of the generalized Heegaard splitting obtained by weak reduction (otherwise, the generalized Heegaard splitting obtained by weak reduction along a weak reducing pair containing the $\W$-disk would be that of the symmetric case of (\ref{as-b}) or (\ref{as-c})), i.e. it cuts off a solid torus from $\W$.
Hence, we can choose a meridian disk $\tilde{W}$ of the solid torus which the $\W$-disk cuts off from $\W$ and it misses three vertices of $\Delta$ by Lemma \ref{lemma-2-8}.
Hence, $\Delta$ and $\tilde{W}$ form a $3$-simplex in $\DVW(F)$.
If the $\W$-disk of $\Delta$ is non-separating, then the boundary of this $\W$-disk must be contained in the genus two component which the boundary of the separating $\V$-disk of $\Delta$ cuts off from $F$ by Lemma \ref{lemma-2-8} and the boundary of the non-separating $\V$-disk of $\Delta$ must be contained in the genus one component by Lemma \ref{lemma-2-9}.
Hence, it is easy to find a band-sum of two parallel copies of the $\W$-disk in $\W$ so that it misses the three disks of $\Delta$, say $W'$, i.e. $\Delta$ and $W'$ form a $3$-simplex in $\DVW(F)$.
In any case, $\Delta$ belongs to a $3$-simplex $\Sigma$ but $\Sigma$ must contain a weak reducing pair consisting of non-separating disks by Lemma \ref{lemma-2-18}.
Since the choice of such a weak reducing pair is unique as $(\bar{V},\bar{W})$ by the previous argument, $\Sigma$ is of the form $\Sigma_{V' W'}$, leading to the result.

If there is a $3$-simplex $\Sigma'$ in $\DVW(F)$, then it must contain a weak reducing pair consisting of non-separating disks by Lemma \ref{lemma-2-18}.
But we can see that this weak reducing pair must be $(\bar{V},\bar{W})$ by the previous argument.
Therefore, $\Sigma'$ is of the form $\Sigma_{V' W'}$, leading to the result.

We don't need to consider more high-dimensional simplex in $\DVW(F)$ by Lemma \ref{lemma-2-18}.
This completes the proof of Claim A.
\end{proofN}

By Claim A, $\DVW(F)=\bigcup_{V', W'} \Sigma_{V' W'}$ for all possible $V'$ and $W'$.
Hence, we can see that $\DVW(F)\cap \DV(F)$ is a $\ast$-shaped graph since every $3$-simplex $\Sigma_{V'W'}$ intersects $\DV(F)$ in an edge and the intersections coming from these $3$-simplices have the common vertex $\bar{V}$.
The symmetric argument also holds for $\DVW(F)\cap\DW(F)$.
Moreover, we can see that if $\Sigma_{V' W'}\neq  \Sigma_{V'' W''}$, then $\Sigma_{V' W'}\cap  \Sigma_{V'' W''}$ is either (i) the weak reducing pair $\{\bar{V},\bar{W}\}$, (ii) the $\V$-face $\{V'=V'', \bar{V},\bar{W}\}$ or (iii) the $\W$-face $\{\bar{V},\bar{W},W'=W''\}$.\\

\ClaimN{B}{$\D(F)$ is contractible.}

\begin{proofN}{Claim B}
If we consider $\D(F)$, then we get 
$$\D(F)=\DV(F)\cup \DVW(F)\cup \DW(F),$$ where the follows hold.
\begin{enumerate}
\item $\DVW(F)\cap \DV(F)$ is a $\ast$-shaped graph,
\item $\DVW(F)\cap \DW(F)$ is a $\ast$-shaped graph, and
\item $\DV(F)\cap \DW(F)=\emptyset$.
\end{enumerate}
In the Chapter 5 in \cite{8}, McCullough proved that $\DV(F)$ and $\DW(F)$ are contractible (Theorem \ref{theorem-mc}) in the sense that they are CW-complexes.
Moreover, we can consider $\DVW(F)$ as a CW-complex.
Recall that $\DVW(F)=\bigcup_{V', W'} \Sigma_{V' W'}$ for all possible $V'$ and $W'$ by Claim A,  where $\Sigma_{V'W'}$ is the $3$-simplex $\{V', \bar{V}, \bar{W}, W'\}$.
Hence, we can construct $\DVW(F)$ from discrete $0$-cells (the vertices of the two $\ast$-shaped graphs $\DVW(F)\cap \DV(F)$ and $\DVW(F)\cap \DW(F)$), followed by $1$-cells (consider the edges of each $\Sigma_{V' W'}$), followed by $2$-cells (consider the faces of each $\Sigma_{V' W'}$), and finally followed by $3$-cells (consider each $\Sigma_{V' W'}$ itself) via attaching maps as in the inductive definition of a CW-complex.

First, we prove that $\DVW(F)$ itself is contractible.
It is sufficient to prove that there is a strong deformation-retraction of $\DVW(F)$ into $\bar{V}$, i.e. a continuous map $h:\DVW(F)\times I \to \DVW(F)$ such that (i) $h(x,0)=x$  for $x\in\DVW(F)$, (ii) $h(x,1)=\bar{V}$ for $x\in\DVW(F)$, and (iii) $h(\bar{V},t)=\bar{V}$ for $0\leq t \leq 1$.

Let us consider the $\ast$-shaped graph $\DVW(F)\cap \DV(F)$.
Recall that $\DVW(F)=\bigcup_{V',W'}\Sigma_{V'W'}$.
If we deformation-retract the $\ast$-shaped graph $\DVW(F)\cap \DV(F)$ into the center point $\bar{V}$ continuously, then each $\Sigma_{V'W'}$ becomes $\{\bar{V},\bar{W},W'\}$ continuously, i.e. $\DVW(F)$ becomes the $\W$-facial cluster containing the weak reducing pair $(\bar{V},\bar{W})$ continuously (see the first arrow of Figure \ref{fig-complex-2}).
We take this process as $h:\DVW(F)\times \left[0,\frac{1}{3}\right]\to \DVW(F)$.
\begin{figure}
	\includegraphics[width=8cm]{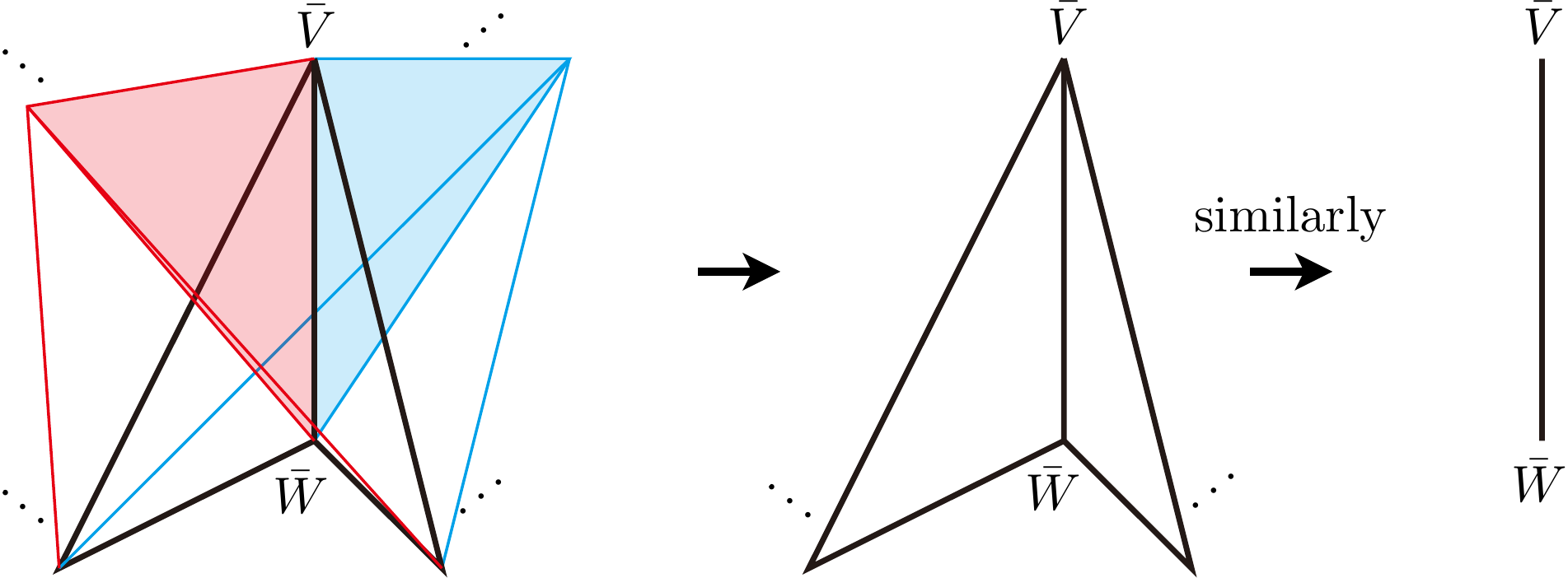}
	\caption{$\DVW(F)$ is contractible.\label{fig-complex-2}}
\end{figure}

Next, if we deformation-retract the $\ast$-shaped graph $\DVW(F)\cap \DW(F)$ into the center point $\bar{W}$ again, then the $\W$-facial cluster containing $(\bar{V},\bar{W})$ becomes $(\bar{V},\bar{W})$ continuously similarly (see the second arrow of Figure \ref{fig-complex-2}).
We take this process as $h:\DVW(F)\times \left[\frac{1}{3},\frac{2}{3}\right]\to \DVW(F)$.

Finally, if we deformation-retract $(\bar{V},\bar{W})$ into $\bar{V}$, then it becomes $\bar{V}$ continuously.
We take this process as $h:\DVW(F)\times \left[\frac{2}{3},1\right]\to \DVW(F)$.\\

Let us introduce the following Lemma.

\begin{lemma}[Exercise 23 of Chapter 0 of \cite{Hatcher2002}]\label{lemma-Hatcher}
A CW complex is contractible if it is the union of two contractible subcomplexes whose intersection is also contractible.
\end{lemma}

Then we can see that $\DV(F)\cup\DVW(F)$ is also a CW-complex.
Moreover, each of $\DV(F)$ and $\DVW(F)$ is a subcomplex of $\DV(F)\cup\DVW(F)$.
Therefore, $\DV(F)\cup\DVW(F)$ is contractible by Lemma \ref{lemma-Hatcher} since $\DV(F)$ is contractible by Theorem \ref{theorem-mc}, $\DVW(F)$ is contractible by the previous observation, and $\DV(F)\cap\DVW(F)$ is a $\ast$-shaped graph which is a contractible one.
Hence, we can see that $\D(F)=(\DV(F)\cup\DVW(F))\cup\DW(F)$ is also contractible similarly by Lemma \ref{lemma-Hatcher}.

This completes the proof of Claim B.
\end{proofN}

\Case{(\ref{as-b})} 
$V$ cuts off $(\text{torus})\times I$ from $\V$ and we can assume that $W$ is non-separating in $\W$.
In this case, $\partial W$ belongs to the genus two component of $F-\partial V$ by Lemma \ref{lemma-2-8}.
Hence, we can take a separating disk $W'\subset \W$ by a band-sum of two parallel copies of $W$ such that $W'$ misses $V$.
Here, we get a $2$-simplex $\Delta=\{V, W, W'\}$ in $\DVW(F)$.
Let $\varepsilon$ be the $\W$-facial cluster containing $\Delta$ guaranteed by Lemma \ref{lemma-2-16}.

If we consider the generalized Heegaard splitting obtained by weak reduction along $(V,W)$, then it consists of two splittings $(\V_1, \V_2;\bar{F}_V)$ and $(\W_1, \W_2;F_W)$ such that $\partial_- \V_2\cap\partial_-\W_1=T$, where $\bar{F}_V$ comes from the genus two component of $F_V$, $\partial_-\V_2$ consists of a torus, and $\partial_-\W_1$ consists of two tori.

Let $\V'$ ($\W'$ resp.) be the solid between $\partial_+\V$ and $\bar{F}_V$ in $\V$ ($\partial_+\W$ and $F_W$ in $\W$ resp.).
Then we can see that $V$ is a separating compressing disk of $\V'$ and $W$ is a non-separating compressing disk of $\W'$ obviously.
Since $\V'$ is a genus three compression body with minus boundary consisting of a torus and a genus two surface and $\W'$ is a compression body with minus boundary consisting of a genus two surface, $V$ and $W$ are uniquely determined in $\V'$ and $\W'$ respectively up to isotopy.
Hence, we get the induced isotopy classes of $V$ and $W$ in $\V$ and $\W$ from the isotopy classes of $V$ and $W$ in $\V'$ and $\W'$.
Moreover, the uniqueness of the isotopy classes of the embeddings of the thick levels of the generalized Heegaard splitting obtained by weak reduction from $(\V,\W;F)$ in the relevant compression bodies forces the choice of the induced isotopy classes of $V$ and $W$ in $\V$ and $\W$ to be unique.
This means that we can consider the weak reducing pair $(V,W)$ as the fixed one $(\bar{V},\bar{W})$ even though we've chosen an arbitrary weak reducing pair consisting of a $\V$-disk cutting off $(\text{torus})\times I$ from $\V$  and a non-separating $\W$-disk.

Let us consider an arbitrary weak reducing pair $(V^\ast, W^\ast)$ for $(\V,\W;F)$.
If $W^\ast$ cuts off $(\text{torus})\times I$ from $\W$, then the generalized Heegaard splitting obtained by weak reduction along $(V^\ast, W^\ast)$ would be the symmetric case of (\ref{as-b}) or (\ref{as-c}), violating the uniqueness of the generalized Heegaard splitting obtained by weak reduction up to isotopy.
Hence, $W^\ast$ does not cut off $(\text{torus})\times I$ from $\W$.
Moreover, if $V^\ast$ does not cut off $(\text{torus})\times I$ from $\V$, then $V^\ast$ is non-separating in $\V$ or it cuts off a solid torus from $\V$, i.e. the generalized Heegaard splitting obtained by weak reduction along $(V^\ast, W^\ast)$ would be that of (\ref{as-a}) or the symmetric case of (\ref{as-b}), leading to a contradiction by the same argument.
Therefore, $V^\ast$ cuts off $(\text{torus})\times I$ from $\V$.
If $W^\ast$ is non-separating in $\W$, then we take $W^{\ast\ast}=W^\ast$.
If $W^\ast$ is separating, then we take $W^{\ast\ast}$ as the meridian disk of the solid torus that $W^\ast$ cuts off from $\W$ and $W^{\ast\ast}\cap V^\ast=\emptyset$ by Lemma \ref{lemma-2-8}.
If we apply the arguments in the previous paragraph to the weak reducing pair $(V^\ast, W^{\ast\ast})$, then  $(V^\ast, W^{\ast\ast})$ would be $(\bar{V},\bar{W})$ and therefore the weak reducing pair $(V^\ast,W^\ast)$  belongs to the $\W$-facial cluster $\varepsilon$.
This means that every weak reducing pair of $(\V,\W;F)$ belongs to the $\W$-facial cluster $\varepsilon$.\\

\ClaimN{C}{
$\DVW(F)=\varepsilon$}

\begin{proofN}{Claim C}
It is sufficient to show that every simplex of $\DVW(F)$ belongs to $\varepsilon$.

By definition of $\DVW(F)$ and the assumption that every weak reducing pair belongs to $\varepsilon$, we don't need to consider vertices or $1$-simplices in $\DVW(F)$.

If we use the same argument in the proof of Claim A, then we can see that if there is a $2$-simplex $\Delta$ in $\DVW(F)$, then it must be a $\V$-face or a  $\W$-face.
If $\Delta$ is a $\V$-face, then a $\V$-disk of $\Delta$ cuts off a solid torus from $\V$ and the other $\V$-disk is a meridian disk of the solid torus by Lemma \ref{lemma-2-9}, i.e. the generalized Heegaard splitting obtained by weak reduction along any weak reducing pair in $\Delta$ would be that of (\ref{as-a}) or the symmetric case of (\ref{as-b}), violating the uniqueness of the generalized Heegaard splitting obtained by weak reduction up to isotopy.

Hence, $\Delta$ must be a $\W$-face.
In this case, we can prove that the $\V$-disk of $\Delta$, say $V'$, must cut off $(\text{torus})\times I$ from $\V$ by the uniqueness of the generalized Heegaard splitting obtained by weak reduction up to isotopy and there is a non-separating $\W$-disk in $\Delta$ by Lemma \ref{lemma-2-9}, say $\tilde{W}$.
If we use the previous argument, then the weak reducing pair $(V',\tilde{W})$ would be $(\bar{V},\bar{W})$ and therefore $\Delta\subset\varepsilon$, leading to the result.

If there is a $3$-simplex $\Sigma$ in $\DVW(F)$, then it must contain a $\V$-face in $\Sigma$ by Lemma \ref{lemma-2-18}, i.e. we get a contradiction similarly as the previous $\V$-face case.

We don't need to consider more high-dimensional simplex in $\DVW(F)$ by Lemma \ref{lemma-2-18}.
This completes the proof of Claim C.
\end{proofN}

We can prove that $\D(F)$ is contractible similarly as Case (\ref{as-a}).\\

\Case{(\ref{as-c})}
Each of $V$ and $W$ cuts off $(\text{torus})\times I$ from $\V$ or $\W$ respectively.

If we consider the generalized Heegaard splitting obtained by weak reduction along $(V,W)$, then it consists of two splittings $(\V_1, \V_2;\bar{F}_V)$ and $(\W_1, \W_2;\bar{F}_W)$ such that $\partial_- \V_2\cap\partial_-\W_1=T$, where $\bar{F}_V$ and $\bar{F}_W$ come from the genus two components of $F_V$ and $F_W$ respectively and both $\partial_-\V_2$ and $\partial_-\W_1$ consist of two tori.

Let $\V'$ ($\W'$ resp.) be the solid between $\partial_+\V$ and $\bar{F}_V$ in $\V$ ($\partial_+\W$ and $\bar{F}_W$ in $\W$ resp.).
Then we can see that $V$ and $W$ are separating compressing disks of $\V'$ and $\W'$ respectively obviously.
Since $\V'$ and $\W'$ are genus three compression bodies with minus boundary consisting of a torus and a genus two surface, $V$ and $W$ are uniquely determined in $\V'$ and $\W'$ respectively up to isotopy.
Hence, we get the induced isotopy classes of $V$ and $W$ in $\V$ and $\W$ from the isotopy classes of $V$ and $W$ in $\V'$ and $\W'$.
Moreover, the uniqueness of the isotopy classes of the embeddings of the thick levels of the generalized Heegaard splitting obtained by weak reduction from $(\V,\W;F)$ in the relevant compression bodies forces the choice of the induced isotopy classes of $V$ and $W$ in $\V$ and $\W$ to be unique.
This means that we can consider the weak reducing pair $(V,W)$ as the fixed one $(\bar{V},\bar{W})$ even though we've chosen an arbitrary weak reducing pair consisting of disks which cut off $(\text{torus})\times I$s in the relevant compression bodies.

Let us consider an arbitrary weak reducing pair $(V^\ast, W^\ast)$ for $(\V,\W;F)$.
If one of $V^\ast$ and $W^\ast$ does not cut off $(\text{torus})\times I$ from $\V$ or $\W$ respectively, then the generalized Heegaard splitting obtained by weak reduction along $(V^\ast, W^\ast)$ would be that of (\ref{as-a}) or (possibly the symmetric case of) (\ref{as-b}), violating the uniqueness of the generalized Heegaard splitting obtained by weak reduction up to isotopy.
Therefore, $V^\ast$ and $W^\ast$ must cut off $(\text{torus})\times I$ from $\V$ and $\W$ respectively.
If we apply the arguments in the previous paragraph to the weak reducing pair $(V^\ast, W^{\ast})$, then  $(V^\ast, W^{\ast})$ would be $(\bar{V},\bar{W})$.
Hence, $\DVW(F)$ is just $(\bar{V},\bar{W})$ itself.

Therefore, $\D(F)$ is contractible obviously.

This completes the proof of Lemma \ref{lemma-3-4}.
\end{proof}

The next lemma deals with the case when the inner thin level consists of two tori.

\begin{lemma}\label{lemma-3-5}
Let $(\V,\W;F)$ be a weakly reducible, unstabilized, genus three Heegaard splitting in an orientable, irreducible $3$-manifold.
If every weak reducing pair of $F$ gives the same generalized Heegaard splitting obtained by weak reduction up to isotopy such that the inner thin level consists of two tori and the embedding of each thick level in the relevant compression body is also unique up to isotopy, then $\D(F)$ is contractible.
\end{lemma}

\begin{proof}
Let us consider a weak reducing pair $(V, W)$ for $(\V,\W;F)$.
Here, $V$ and $W$ are non-separating but $\partial V\cup\partial W$ is separating in $F$ as we checked in the proof of Lemma \ref{lemma-2-17} by the assumption that the inner thin level consists of two tori.
This means that $\partial V\cup \partial W$ cuts off two twice-punctured tori from $F$.
If we consider the generalized Heegaard splitting obtained by weak reduction along $(V,W)$, then it consists of two splittings $(\V_1, \V_2;F_V)$ and $(\W_1, \W_2;F_W)$ such that $\partial_- \V_2=\partial_-\W_1=F_{VW}=T_1\cup T_2$, where each $T_i$ is a torus.

Let $\V'$ ($\W'$ resp.) be the solid between $\partial_+\V$ and $F_V$ in $\V$ ($\partial_+\W$ and $F_W$ in $\W$ resp.).
Then we can see that $V$ and $W$ are non-separating compressing disks of $\V'$ and $\W'$ respectively obviously.
Since $\V'$ and $\W'$ are genus three compression bodies with minus boundary consisting of a genus two surface, $V$ and $W$ are uniquely determined in $\V'$ and $\W'$ respectively up to isotopy.
Hence, we get the induced isotopy classes of $V$ and $W$ in $\V$ and $\W$ from the isotopy classes of $V$ and $W$ in $\V'$ and $\W'$.
Moreover, the uniqueness of the isotopy classes of the embeddings of the thick levels of the generalized Heegaard splitting obtained by weak reduction from $(\V,\W;F)$ in the relevant compression bodies forces the choice of the induced isotopy classes of $V$ and $W$ in $\V$ and $\W$ to be unique.
This means that we can consider the weak reducing pair $(V,W)$ as the fixed one $(\bar{V},\bar{W})$ even though we've chosen an arbitrary weak reducing pair.
Hence, $\DVW(F)$ is just $(\bar{V},\bar{W})$ itself.

Therefore, $\D(F)$ is contractible similarly as Case (\ref{as-c}) of Lemma \ref{lemma-3-4}.
This completes the proof.
\end{proof}

Therefore, Lemma \ref{lemma-3-1}, Lemma \ref{lemma-new}, Lemma \ref{lemma-3-4} and Lemma \ref{lemma-3-5} give Theorem \ref{theorem-main-a}.

\begin{theorem}\label{theorem-main-a}
Let $(\V,\W;F)$ be a weakly reducible, unstabilized, genus three Heegaard splitting in an orientable, irreducible $3$-manifold $M$.
If every weak reducing pair of $F$ gives the same generalized Heegaard splitting after weak reduction and the embedding of each thick level in the relevant compression body is also unique up to isotopy, then the disk complex $\D(F)$ is contractible.
Otherwise, $F$ is critical.
\end{theorem}

\section*{Acknowledgments}
This research was supported by BK21 PLUS SNU Mathematical Sciences Division.

\appendix

\section{The pseudo-GHSs obtained by pre-weak reductions\label{appendix}}

In this section, we give descriptive figures of all possible cases of the pseudo-GHSs obtained by pre-weak reductions from an unstabilized, genus three Heegaard splitting $(\V,\W;F)$ of an irreducible $3$-manifold $M$, where these pre-weak reductions give GHSs of the form $(\V_1,\V_2)\cup (\W_1,\W_2)$ for $\partial_-\V_2\cap\partial_-\W_1\neq\emptyset$.
Note that we only consider the case $\partial_- \V_1=\partial_-\W_2=\emptyset$, since $\partial_- \V_1$ and $\partial_-\W_2$ do not affect the shape of $\V_2$ and $\W_1$ and the inner thin level.

\begin{enumerate}
\item The pseudo-GHSs give the GHS in (a) of Figure \ref{figure-5}, see Figure \ref{fig-wrca}. (We omit the symmetric case for the second figure.)
\begin{figure}
	\includegraphics[width=13cm]{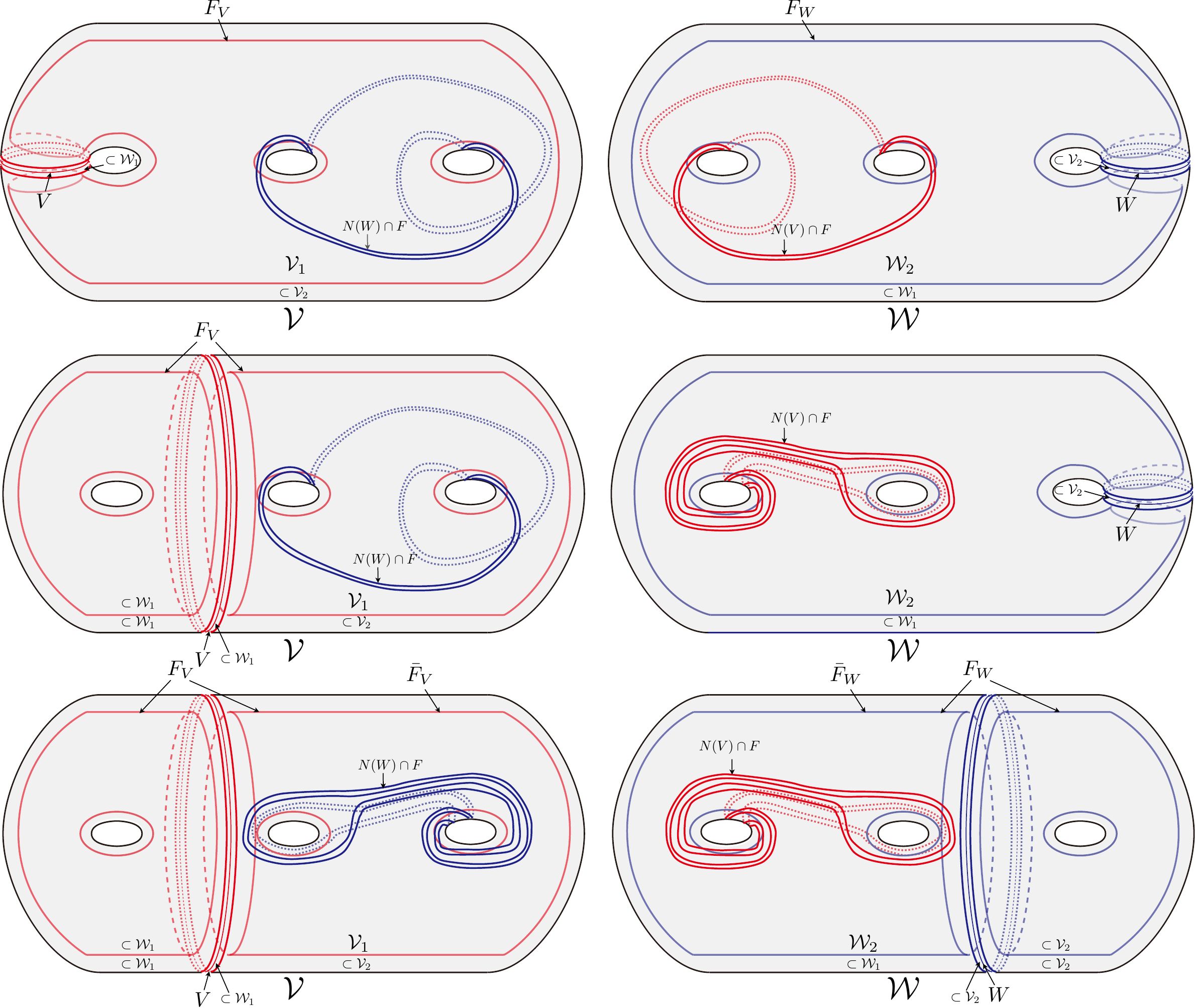}
	\caption{The pseudo-GHSs for (a) of of Figure \ref{figure-5}\label{fig-wrca}}
\end{figure}
\item The pseudo-GHSs give the GHS in (b) of Figure \ref{figure-5}, see Figure \ref{fig-wrcb}. (We omit the symmetric cases when $W$ cuts off $(\text{torus})\times I$.)
\begin{figure}
	\includegraphics[width=13cm]{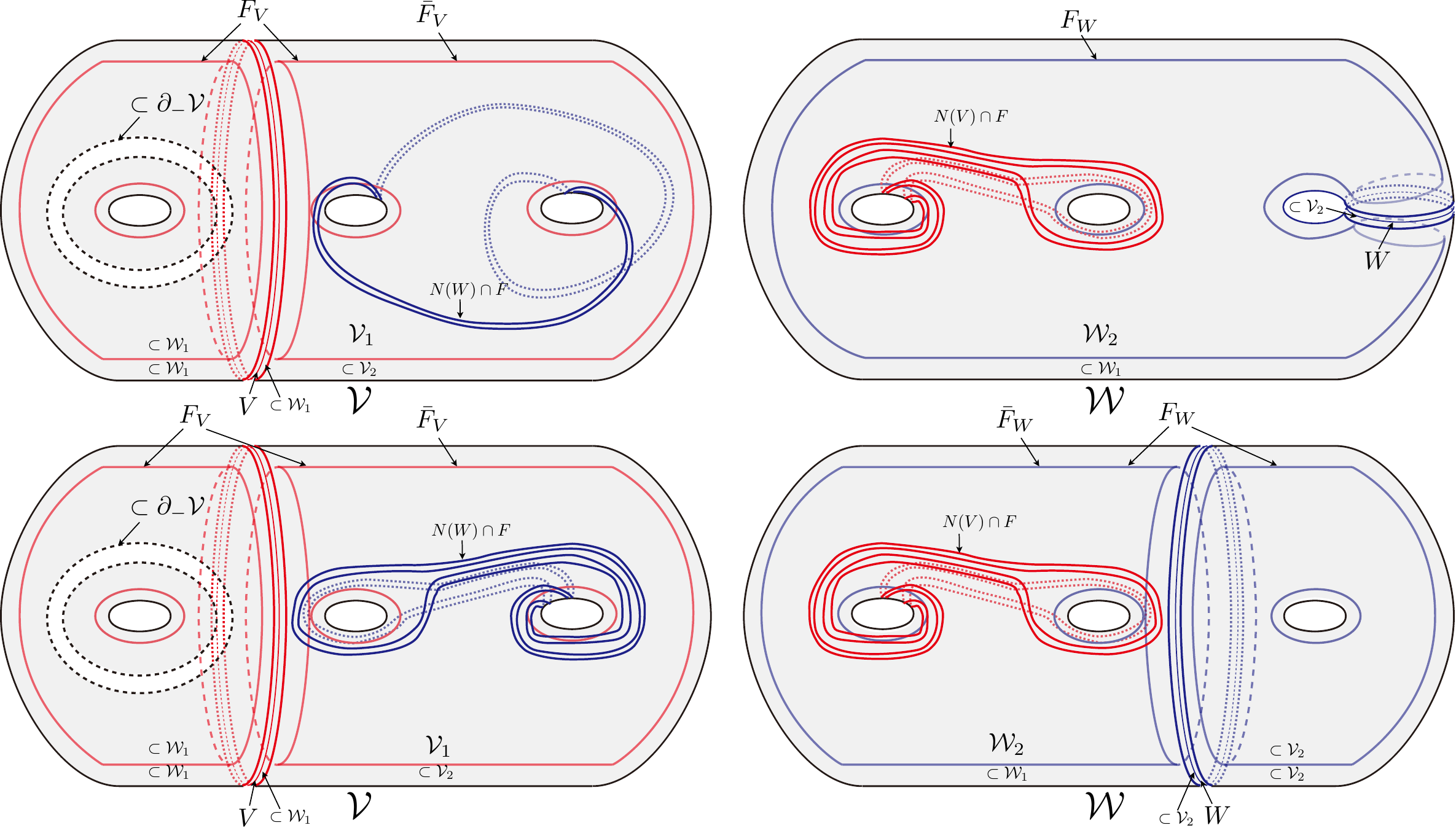}
	\caption{The pseudo-GHSs for (b) of of Figure \ref{figure-5}\label{fig-wrcb}}
\end{figure}
\item The pseudo-GHS gives the GHS in (c) of Figure \ref{figure-5}, see Figure \ref{fig-wrcc}.
\begin{figure}
	\includegraphics[width=13cm]{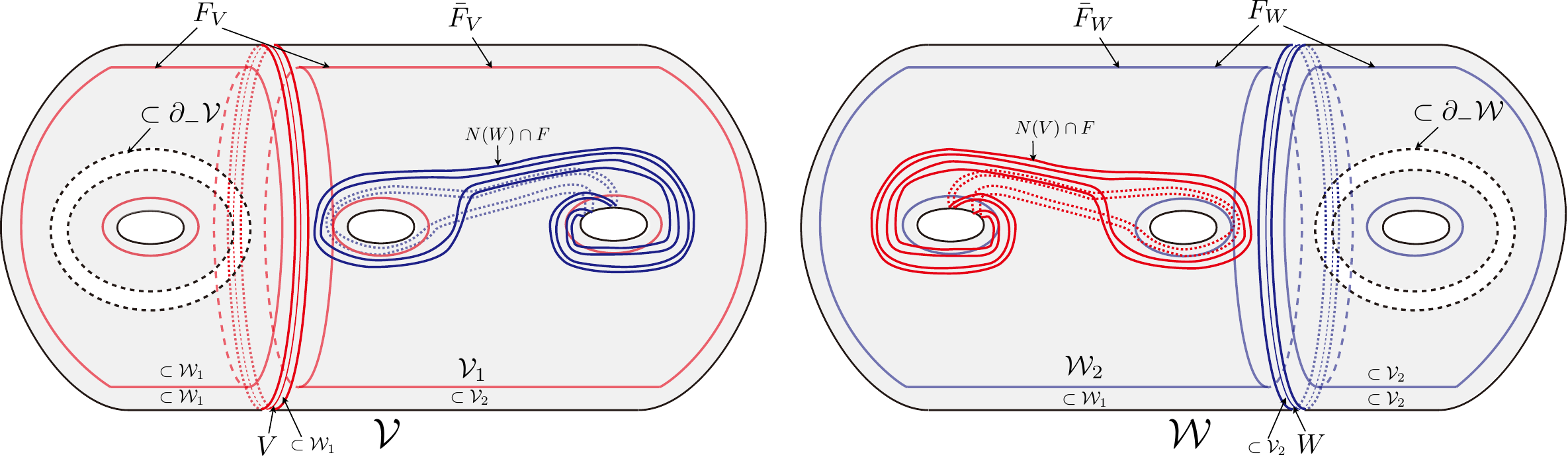}
	\caption{The pseudo-GHS for (c) of of Figure \ref{figure-5}\label{fig-wrcc}}
\end{figure}
\item The pseudo-GHS gives the GHS Figure \ref{figure-4}, see Figure \ref{fig-wrcd}.
\begin{figure}
	\includegraphics[width=13cm]{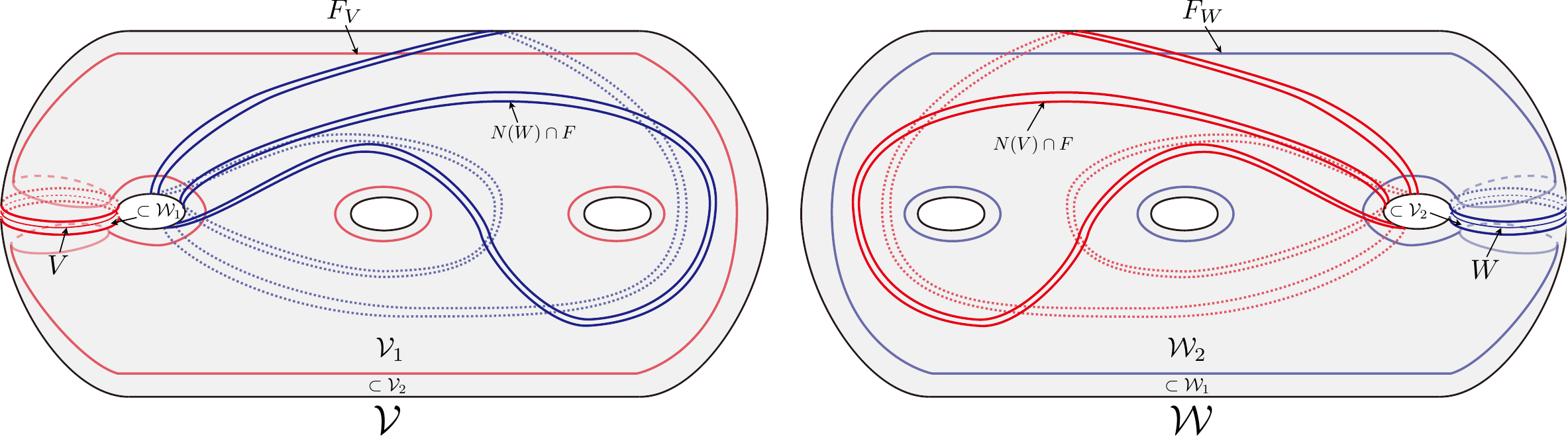}
	\caption{The pseudo-GHS for Figure \ref{figure-4}\label{fig-wrcd}}
\end{figure}
\end{enumerate}
\end{document}